\documentclass[a4paper,11pt]{article}
\usepackage[utf8]{inputenc}

\usepackage{verbatim}
\usepackage{amsfonts}
\usepackage{a4wide}
\usepackage{amsthm}
\usepackage[dvipsnames]{xcolor}
\usepackage{graphicx}
\usepackage{amsmath,amssymb}
\usepackage{fancyhdr}
\usepackage{hyperref}
\usepackage{doi}
\usepackage[normalem]{ulem}
\usepackage[varbb]{newtxmath}  
\usepackage{relsize}
\usepackage{enumerate}
\usepackage{mathtools, stackengine}  
\usepackage{bm}  

\newcommand{\mres}{\mathbin{\vrule height 1.6ex depth 0pt width
0.13ex\vrule height 0.13ex depth 0pt width 0.9ex}}  

\usepackage[square,numbers]{natbib} 

\newcommand{\Lb}{\mathcal L}
\newcommand{\N}{\mathbb N}
\newcommand{\E}{\mathcal E}
\newcommand{\F}{\mathcal F}
\newcommand{\dd}{\,\mathrm d}

\newcommand{\argmin}{\mathrm{arg\,min}\,}
\newcommand{\eps}{\varepsilon}
\newcommand{\R}{\mathbb R}

\newcommand{\Id}{\mathrm{Id}}
\newcommand{\indi}{\mathlarger{\mathlarger{\mathbb{1}}}}  
\newcommand{\ourphi}{\widetilde{\varphi}_\eta}

\newcommand{\weaklystar}{\mathrel{\ensurestackMath{\stackon[0.5pt]{\rightharpoonup}{\scriptstyle\ast}}}}

\newtheorem{prop}{Proposition}[section]
\newtheorem{thm}[prop]{Theorem}
\newtheorem{lemma}[prop]{Lemma}
\newtheorem{defn}[prop]{Definition}
\newtheorem{cor}[prop]{Corollary}
\theoremstyle{remark}
\newtheorem{remark}[prop]{Remark}

\newcommand{\ignore}[1]{{}}
\usepackage{titlesec}
\titlelabel{\thetitle.\quad}
\author{Zofia Grochulska$^*$, Micha{\l} {\L}asica$^\dagger$
	\\
\small $^*$Institute of Mathematics, University of Warsaw\\
{\small \tt z.grochulska@uw.edu.pl} \\
\small $^\dagger$Institute of Mathematics of the Polish Academy of Sciences \\
\small \& Graduate School of Mathematical Sciences, The University of Tokyo \\
{\tt \small mlasica@impan.pl}
}

\title{Local estimates for vectorial Rudin--Osher--Fatemi type problems in one dimension} 

\date{\today}

\begin{document}
 \maketitle
\begin{abstract}
    We consider the Rudin--Osher--Fatemi variational denoising model with general regularizing term in one-dimensional, vector-valued setting. We obtain local estimates on the singular part of the variation measure of the minimizer in terms of the singular part of the variation measure of the datum. In the case of homogeneous regularizer, we prove local estimates on the whole variation measure of the minimizer and deduce an analogous result for the gradient flow of the regularizer. 
\end{abstract} 
\section{Introduction}

Let $\Omega \subset \R^m$ be an open, bounded domain with Lipschitz boundary and let $F\colon \R^{m \times n} \to [0, \infty[$ be a~convex function. We consider the functional $\widetilde{\F} \colon C^1(\Omega, \R^n)\to [0, \infty]$ given by
\begin{equation} \label{tildeF} 
	\widetilde{\F}(w) = \int_{\Omega} F(w_x) \dd \Lb^m, 
\end{equation} 
where $w_x$ denotes the (total) derivative of $w$, and $\F \colon L^2(\Omega, \R^n) \to [0, \infty]$ defined as a lower semicontinuous envelope of $\widetilde{\F}$, i.\,e. 
\begin{equation}\label{L2_env} 
	\F(w) = \inf\left\{\liminf_{\ell \to \infty} \widetilde{\F}(w^\ell), \ (w^\ell) \subset C^1(\Omega, \R^n),\ w^\ell \to w \text{ in } L^1_{loc}(\Omega, \R^n)  \right\}. 
\end{equation} 
We define a generalized Rudin--Osher--Fatemi functional $\E \colon L^2(\Omega, \R^n) \to [0, \infty]$ by 
   \begin{equation} \label{ROF_def} 
   	 \E_h^\lambda(w) = \lambda \F(w) + \frac{1}{2} \int_{\Omega} |w - h|^2 \dd \Lb^m, 
   \end{equation}
where $h \in L^2(\Omega, \R^n)$, $\lambda > 0$. We will drop indices $h, \lambda$ as long as we are not interested in dependence on them. By convexity, this functional is weakly lower semicontinuous. Moreover, it is proper and its sublevel sets are weakly compact. Therefore it admits a minimizer $u$, which is unique since $\E$ is strictly convex on its domain (i.\,e.\ the subset of $L^2(\Omega, \R^n)$ where $\E$ is finite).

In the superlinear case, where $F(\xi) \sim |\xi|^p$, $p > 1$, the domain of $\F$ (and $\E$) coincides with $L^2(\Omega, \R^n)\cap W^{1,p}(\Omega, \R^n)$, and on this set $\F$ is given by formula \eqref{tildeF}. However, here we consider $F$ of \emph{linear growth}, i.\,e. 
\begin{equation} \label{lin_growth} 
	C_F^- |\xi| + B_F^- \leq F(\xi) \leq C_F^+ |\xi| + B_F^+ \quad \text{for } \xi \in \R^{m \times n}
\end{equation}
with $C_F^\pm > 0$, $B_F^\pm \in \R$. In this case, analogous statement does not hold, which is related to lack of reflexivity of $W^{1,1}(\Omega, \R^n)$, and the domain of $\F$ coincides with $L^2(\Omega, \R^n)\cap BV(\Omega, \R^n)$. Recall that $BV(\Omega, \R^n)$ is the space of those elements $w$ of $L^1(\Omega, \R^n)$ whose distributional derivative $w_x$ belongs to $M(\Omega, \R^{m \times n})$, the space of vector-valued Radon measures. An explicit characterization of $\F$ was obtained by Goffman and Serrin in \cite{GS64} and is described in more detail in Section 2. Given a~$BV$ mapping $w$, one can write decomposition of its derivative $w_x$ into absolutely continuous and singular parts with respect to the Lebesgue measure as $\dd w_x = w_x^{ac} \dd  \Lb^m + \dd w_x^s$. Consequently, one can define a measure $F(w_x) \in M(\Omega)$ by
\begin{equation} \label{Fwx}
	\dd F(w_x) = F(w_x^{ac}) \dd \Lb^m + F^{\infty}\left(\tfrac{w_x^s}{|w_x^s|}\right) \dd |w_x^s|, \quad \text{where } F^{\infty}(\nu) = \lim_{t \to \infty} \tfrac{1}{t}F(t \nu) \text{ for } \nu \in \mathbb S^{n-1}.
\end{equation}  
Then, one can prove that for $w \in BV(\Omega, \R^n)$
\begin{equation}
    \F(w) = F(w_x)(\Omega).
\end{equation}
In the case $F = |\cdot|$, $\F$ coincides with the standard $BV$ seminorm also known as the total variation, while $\E$ coincides with the classical Rudin--Osher--Fatemi functional, see \cite{rof, ChambolleLions}.

A natural question arises, what can be said about the singular part of $u_x$. So far, this question has been mostly investigated in the scalar-valued case $n=1$. Let us report briefly on known results in this direction. Under the assumption that $\Omega$ is convex, it has been proved in \cite{LasicaRybka} that 
\[F^\infty(u_x^s)(\Omega) \leq F^\infty(h_x^s)(\Omega).\]
In particular, if $h \in W^{1,1}(\Omega)$, then $u \in W^{1,1}(\Omega)$. On the other hand, in the case where $F$ is the $\ell^1$ norm, one can construct a non-convex polygon $\Omega \subset \R^2$ and $h \in C^\infty(\overline{\Omega})$ such that $u$ has a jump discontinuity in $\Omega$ \cite{tetris}. This sort of singularity does not appear if $F$ is the Euclidean norm, in which case $|u_x^j| \leq |h_x^j|$ as measures, i.\,e.
\[|u_x^j|(V) \leq |h_x^j|(V) \quad \text{for Borel } V \subset \Omega\]
for $h \in BV(\Omega) \cap L^\infty(\Omega)$, where $w_x^j$ denotes the jump part of $w_x$ for $w \in BV(\Omega)$ \cite{jalalzaijump}. This estimate has been generalized to a wider class of functions $F$ in \cite{Valkonen2015}. To our knowledge, it remains an open question whether an analogous estimate holds for the whole singular part of $u_x$. Moreover, no similar estimate on the singular or jump part of $u_x$ is known in the vector-valued setting $n > 1$, except in the one-dimensional case $m=1$.  

From now on we will restrict ourselves to $m=1$. In the one-dimensional, scalar-valued case $m=n=1$, for $F = |\cdot|$, it has been proved in \cite{bcno, BonforteFigalli} that
\begin{equation} \label{1d_homo_est}  
	|u_x| \leq |h_x| \quad \text{as measures.}
\end{equation} 
The proof in \cite{bcno} follows by analysis of level sets of $u$, thus depending on the linear order of the range $\R$. On the other hand, the reasoning in \cite{BonforteFigalli} is based on explicit description of minimizers for piecewise constant $h$. Inequality \eqref{1d_homo_est} was generalized to the vector-valued case $n > 1$ in \cite{giacomellilasica} using integral estimates.  

In this paper, our goal is to obtain estimates on the singular part of the minimizer in the case $m=1$, $n>1$ for possibly general $F$. We will need a structural assumption formally similar to the one in \cite{mercier} (note however that in \cite{mercier} $m>1$, $n=1$ and the results are not closely related). We recall the following definition (compare e.\,g.\ \cite{ChambolleNovaga2015}). We say that a function $\varphi \colon \R^n \to [0, +\infty[$ is an \emph{anisotropy} if it is convex and positively $1$-homogeneous, i.\,e. 
	\[\varphi(t p) = t \varphi(p) \quad \text{for } p \in \R^n,\ t \geq 0.\]
	For any anisotropy $\varphi$, we denote 
	\[c_{\varphi}^+ = \max_{0 \neq p \in \R^n} \frac{\varphi(p)}{|p| }, \quad c_{\varphi}^- = \min_{0 \neq p \in \R^n} \frac{\varphi(p)}{|p|}. \]
An anisotropy $\varphi$ will be called \emph{coercive} if $c_\varphi^- > 0$.  Note that any even anisotropy is a seminorm (and vice versa) and any coercive, even anisotropy is a norm (and vice versa). We prove the following three local estimates on the variation measure of minimizers of $\E$ in terms of the variation measure of the datum $h$ under varying assumptions on $F$.
\begin{thm} \label{thm:homo}
	Let $F = \varphi \colon \R^n \to [0, \infty[$ be a coercive anisotropy. Suppose that $h \in BV(U, \R^n)$ for an open interval $U \subset I$. Then the minimizer $u$ of $\E$ satisfies $|u_x| \leq |h_x|$ in the sense of Borel measures on $U$, i.\,e.
	\begin{equation} \label{homo_est} 
		|u_x|(E) \leq |h_x|(E) \text{ for any Borel set } E \subset U.
	\end{equation}
\end{thm}

\begin{thm} \label{thm:constants}
	Let $F = f \circ \varphi$, where $f\colon[0, \infty[ \to [0, \infty[$ is convex, non-decreasing and of linear growth, and $\varphi\colon \R^n \to [0, \infty[$ is a coercive anisotropy. Suppose that $h \in BV(U, \R^n)$ for an open interval $U \subset I$. Then the minimizer $u$ of $\E$ satisfies $|u_x^{s}| \leq \left( \frac{c_\varphi^+}{c_\varphi^-} \right)^2|h_x^{s}|$ in the sense of Borel measures on $U$, i.\,e.
	\begin{equation} \label{eq:constants}
		|u_x^{s}|(E) \leq \left( \tfrac{c_\varphi^+}{c_\varphi^-} \right)^2 |h_x^{s}|(E) \text{ for any Borel set } E \subset U.
	\end{equation}
\end{thm}

\begin{thm} \label{thm:regular norm} 
	Let $F = f \circ \varphi$, where $f\colon [0, \infty[ \to [0, \infty[$ is strictly convex, increasing and of linear growth, and $\varphi\colon \R^n \to [0, \infty[$ is a~coercive anisotropy which is strictly convex in the sense of Reshetnyak (see Definition \ref{def: Reshetnyak}). Moreover, assume that $\varphi$ is differentiable in $\R^n \setminus \{0\}$ and that $f$ is differentiable with $f'(0) = 0$. Suppose that $h \in BV(U, \R^n)$ for an open interval $U \subset I$. Then the minimizer $u$ of $\E$ satisfies $|u_x^{s}| \leq |h_x^{s}|$ in the sense of Borel measures on $U$, i.\,e.,
	\begin{equation} \label{eq: regular norm}
		|u_x^{s}|(E) \leq |h_x^{s}|(E) \text{ for any Borel set } E \subset U.
	\end{equation}
\end{thm}

\noindent
Note that under the assumptions of Theorems \ref{thm:homo}, \ref{thm:constants} and \ref{thm:regular norm}, $F$ is indeed a convex function of linear growth.

The basic structure of the proofs of the three theorems is similar---a~suitable two-level approximation $\E_{\eps, \eta}$ of the functional $\E$ is used. On the more regular level, the minimizer $u^{\eps, \eta}$ satisfies the Euler--Lagrange equation 
\[u^{\eps, \eta} - h^\eps = \lambda DF_\eta(u_x^{\eps, \eta})_x \text{ in } I, \quad DF_\eta(u_x^{\eps, \eta}) = 0 \text{ on } \partial I \]
in strong sense. We test the equation with a suitable function of the form $(\zeta^2 G(u_x^{\eps, \eta}))_x$, where $\zeta \in C^1_c(I)$, while $G$ is chosen so that $G(w_x)_x \cdot DF_\eta(w_x)_x \geq 0$ for $w \in W^{2,2}(I, \R^n)$ and $G(\xi) \cdot \xi \sim |\xi|$ for large $|\xi|$. By virtue of the latter property, after integration by parts on the l.\,h.\,s., a quantity of order $|u_x^{\eps, \eta}|$ appears. It is essential for our argument to take $G$ of the form $G(\xi) = g(\xi) \xi$ with $g \colon \R^n \to \R$. Due to such a~choice, the bulk of terms involving $\zeta_x$ vanishes, owing to equality $\xi \cdot D^2\varphi(\xi) = 0$, which follows from homogeneity of $\varphi$. In the homogeneous case $F = \varphi$, performing this procedure with a relatively simple $g$ readily leads to the strongly local estimate \eqref{homo_est}. 

In the inhomogeneous case $F = f \circ \varphi$, the argument becomes more involved and, expectedly, yields a bound only on the singular part of $u_x$. One of essential difficulties in the proof of such estimates is posed by very weak continuity properties of the operator $w \mapsto w_x^s$ on $BV(I, \R^n)$. Notably, maps such as $w \mapsto |w_x^s|(\Omega)$ fail to be semicontinuous w.\,r.\,t.\ weak* as well as strict or area-strict convergence in $BV(I, \R^n)$. Therefore, we need to work instead with quantities of type 
\[ (|w_x| - k)_+ =  \int (|w_x^{ac}| - k)_+ \dd \Lb^1 +  |w_x^s| \]
and obtain estimates for their decay as $k \to \infty$. However, our previously described procedure delivers estimates for more complicated, non-convex functions of $u_x$ which again fail to be lower semicontinuous.    
Nonetheless, by exploiting equivalence between $\varphi$ and $|\cdot|$, we can deduce estimates for a convex function of $u_x$, but only up to a~multiplicative constant.

We note that such a constant is undesirable in the context of image processing, since in applications the minimization procedure tends to be iterated many times (in the spirit of the  \emph{minimizing movements scheme}---we recall that notion in the next paragraph). In the case when $F$ is strictly convex and differentiable, we obtain a better estimate \eqref{eq: regular norm}. In the proof we use spherical compactification of $\R^n$ (see Section~\ref{sec: prelim}) and the fact that the derivative $DF$ is a~homeomorphism to improve the convergence of approximate minimizers. This is one of the reasons why we need additional assumptions on $F$ in this case. However, we do not know whether \eqref{eq: regular norm} can fail otherwise.

Let us recall that the convex, lower semicontinuous function $\F$ on the Hilbert space $L^2(I, \R^n)$ defines a unique gradient flow \cite{maximaux}. In other words, for any $v_0 \in L^2(I, \R^n)$ there exists a unique function $v \in C([0, \infty[, L^2(I, \R^n))\cap W^{1,2}_{loc}(0,T;L^2(I, \R^n))$ such that 
\begin{equation} \label{grad_flow} 
	v_t(t) \in - \partial \F(v(t)) \quad \text{ for a.\,e.\ } t > 0, \qquad v(0) = v_0.  
\end{equation}  
In the case $F = |\cdot|$ this is a vectorial version of the so-called \emph{total variation flow}, formally given by the equation 
\[v_t = \left(\frac{v_x}{|v_x|}\right)_x \text{ in } ]0, \infty[ \times I, \quad v = 0 \text{ on } ]0, \infty [ \times \partial I,\]
cf.\ \cite{acmbook}. The map associating the minimizer of $\E_h^\lambda$ to a given $h$ coincides with the resolvent operator for the subdifferential $- \partial \F$. This is the basis of the \emph{minimizing movements scheme}: for a~given $N \in \N$ we iteratively define \[v^N_0 = v_0, \qquad v^N_j = \argmin E_{v^N_{j-1}}^{1/N} \quad \text{for } j=1,2,\ldots\]
Then, for a.\,e.\ $t > 0$ we have 
\begin{equation} \label{grad_flow_conv} 
	v^N_j \to v(t) \quad \text{in } L^2(I, \R^n) \quad \text{if } N \to \infty,\ j/N \to t.
\end{equation}     
Let $F=\varphi$ be a~coercive anisotropy. Then, by Theorem \ref{thm:homo}, for all $N,j \in \N$, it is true that $v^N_j \in BV(U,\R^n)$ and 
\begin{equation*} 
	|(v^N_j)_x |(V) \leq |(v^N_{j-1})_x |(V) \leq \ldots \leq |(v_0)_x |(V)
\end{equation*} 
for any open $V \subset U$. Using \eqref{grad_flow_conv} and Theorem \ref{thm-gs-lsc}, we deduce $|v_x(t)|(V) \leq |(v_0)_x|(V)$ for any open $V \subset U$ and a.\,e.\ $t>0$. Thus, we obtain
\begin{cor}\label{thm: grad flows} 
	Suppose that $F = \varphi$ is a coercive anisotropy. If $v_0 \in BV(U, \R^n)$ with open $U \subset I$, then $|v_x(t)| \leq |v_{0,x}|$ in the sense of Borel measures on $U$, i.\,e.
	\begin{equation}\label{gf_thm_1}  
	|v_x(t)|(E) \leq |v_{0,x}|(E) \quad \text{for any Borel set } E \subset U.
	\end{equation} 
\end{cor}

This generalizes an analogous result for the scalar 1D total variation flow from \cite{bcno, BonforteFigalli}. We note also a recent paper \cite{MollSmarrazzo2022}, where similar estimate was obtained for more general parabolic equations in the scalar 1D case $m=n=1$. Moreover, the authors provide conditions for instantaneous regularization $BV \to W^{1,1}_{loc}$ and $L^1 \to W^{1,1}_{loc}$.

We are unable to transfer Theorems \ref{thm:constants} and \ref{thm:regular norm} to the gradient flow setting since $|u_x^s|$ does not have good semicontinuity properties. On the other hand, intermediate estimates \eqref{eq:constants final} and \eqref{eq: before k limit} do not behave well under iteration. 

In the next section we recall some facts that will be used in the sequel. In Section~\ref{sec: approximate} approximate functionals are defined and in Sections \ref{sec:homo}, \ref{sec:coercive}, \ref{sec:regular} we prove Theorems \ref{thm:homo}, \ref{thm:constants} and\ref{thm:regular norm}, respectively. 

\section{Preliminaries} \label{sec: prelim}

Let $d$ be any positive integer. Throughout the paper, $\varrho \in C^\infty_c(\R^d)$ denotes a~positive, radially symmetric function such that $\int_{\R^d} \varrho \dd \Lb^d = 1$ and $\varrho_\eps(p) = \eps^{-d} \varrho(p/\eps)$. We will say that measures $\mu_j \in M(\Omega, \R^d)$ converge weakly* to $\mu \in M(\Omega, \R^d)$ and write $\mu_j \weaklystar \mu$, if for any $\zeta \in C_c(\Omega)$ it is true that $\lim_{j \to \infty} \int_\Omega \zeta \dd \mu_j = \int_\Omega \zeta \dd \mu$.

\subsection*{Functions of measures}

Definition \eqref{Fwx} from the Introduction can be applied to any Radon measure, not necessarily a~derivative of a~$BV$ function. Throughout this subsection, $G$ denotes~a convex function $G\colon \R^d \to [0, \infty[$ of linear growth. Given such a~function and a~measure $\mu \in M(\Omega, \R^d)$, we define measure $G(\mu)$ by 
\begin{equation} \label{Fmu}
	\dd G(\mu) = G(\mu^{ac}) \dd \Lb^m + G^{\infty}\left(\tfrac{\mu^s}{|\mu^s|}\right) \dd |\mu^s|, \quad \text{where } G^{\infty}(\nu) = \lim_{t \to \infty} \tfrac{1}{t}G(t \nu) \text{ for } \nu \in \mathbb S^{n-1}
\end{equation}
and $\mu^s/ |\mu^s|$ denotes the Radon--Nikodym derivative. The function $G^\infty$ is often called \emph{the~recession function} of $G$. In the following paragraphs, facts about properties of $G(\mu)$ are collected.

\begin{thm}\cite[p.\ 172]{GS64} \label{thm-gs-lsc}
	Let $\mu_j, \mu \in M(\Omega, \R^d)$ and assume that $\mu_j \weaklystar \mu$. Then for any open $U \subset \Omega$,
	\begin{equation*}
		G(\mu)(U) \leq \liminf_{j \to \infty} G(\mu_j)(U).
	\end{equation*}
\end{thm}

\begin{cor} \label{cor-gs-lsc-zeta}
	Let $\mu_j, \mu \in M(\Omega, \R^n)$ and assume that $\mu_j \weaklystar \mu$. Then for any non-negative $\zeta \in C_c(\Omega)$,
	\begin{equation*}
		\int_\Omega \zeta \dd G(\mu) \leq \liminf_{j \to \infty} \int_\Omega \zeta \dd G(\mu_j).
	\end{equation*}
\end{cor}

\begin{proof}
	Suppose that $\liminf_{j \to \infty} \int_\Omega \zeta \dd G(\mu_j) < + \infty$ (if not, the inequality clearly holds). Choose a~subsequence $\mu_{j_\ell}$ such that
\begin{equation*}
	\lim_{\ell \to \infty} \int_\Omega \zeta \dd G(\mu_{j_\ell}) = \liminf_{j \to \infty} \int_\Omega \zeta \dd G(\mu_j).
\end{equation*}
Since $G$ is of linear growth, the sequence of measures $\left \{ G(\mu_{j_\ell}) \right \}_\ell$ is bounded and hence there exists a~subsequence (which we do not relabel) weakly* convergent to a~measure $\nu$. Both $\nu$ and $G(\mu)$ are finite on $\Omega$. Therefore, at each point $x \in \Omega$ there exists a~family of open balls $\left \{ V_x^i \right \}_{i=1}^{N(x)}$ centered at $x$ with $\mathrm{diam}\left ({V_x^i} \right) \to 0$ as $i \to \infty$, for which 
\[
	\nu(\partial V) = G(\mu)(\partial V) = 0.
\]

The family $\left \{ \overline{V_x^i}\colon \, x \in \Omega, \, i = 1, \ldots N(x) \right \}$ of closed balls satisfies assumptions of the~corollary to Besicovitch's covering theorem \cite[Section 1.5, Corollary I]{evansgariepy}. For any open $U \subset \Omega$, it is then possible to choose a~countable subfamily $\mathcal{G}$ of disjoint balls such that
\begin{equation} \label{eq: family G}
	\mathsmaller{\bigcup} \mathcal{G} \subset U \text{ and } \nu\left( U \setminus \mathsmaller{\bigcup} \mathcal{G} \right) + G(\mu)\left( U \setminus \mathsmaller{\bigcup} \mathcal{G} \right) = 0.
\end{equation}
For any ball $V$ with $\overline{V} \in \mathcal{G}$, it is true that $G(\mu_{j_\ell})(V) \to \nu(V)$ and using Theorem \ref{thm-gs-lsc}, we may write
\begin{equation*}
	\nu (\overline{V}) = \nu(V) = \lim_{\ell \to \infty} G(\mu_{j_\ell})(V) \geq G(\mu)(V) = G(\mu)(\overline{V}).
\end{equation*}
Summing over balls from the family $\mathcal{G}$ and using \eqref{eq: family G} results in the estimate $\nu(U) \geq G(\mu)(U)$ for any open $U \subset \Omega$, which implies that
\[ 
\nu(E) \geq G(\mu)(E) \text{ for any Borel } E \subset \Omega.
\]
As a result, in view of the choice of the subsequence $\mu_{j_\ell}$, one gets
\[
\int_\Omega \zeta \dd G(\mu) \leq \int_\Omega \zeta \dd \nu = \lim_{\ell \to \infty} \int_\Omega \zeta \dd G(\mu_{j_\ell}) = \liminf_{j \to \infty} \int_\Omega \zeta \dd G(\mu_j).
\]
\end{proof}

\begin{lemma}\cite[p.\ 171]{GS64} \label{lem: jensen GS}
Let $U \subset \R^m$ be an open set and $b\colon U \to [0, \infty[$ an integrable function with $\int_U b \dd \Lb^m \neq 0$. For $\mu \in M(U, \R^d)$, it is true that
\[
	G \left( \frac{\int_U b \dd \mu}{\int_U b \dd \Lb^m} \right) \leq \frac{\int_U b \dd G(\mu)}{\int_U b \dd \Lb^m}. 
\]
\end{lemma}

\begin{lemma} \label{lem: G weaklystar}
Suppose that $w \in L^1(\Omega) \cap BV(U)$ for some open set $U$ compactly contained in $\Omega$ and let $w^\eps = w \ast \varrho_\eps$. Then, $G(w_x^\eps) \weaklystar G(w_x)$ in $M(U)$.
\end{lemma}

\begin{proof}
	The statement follows from a small modification of \cite[Theorem 2.2]{ambrosiofuscopallara}, which requires the use of Lemma \ref{lem: jensen GS}.
\end{proof}

The following theorem establishes coincidence of $G(w_x)(\Omega)$ with the relaxation of the integral functional $w \mapsto \int_\Omega G(w_x) \dd \Lb^m$, as mentioned in the Introduction. 
\begin{thm}\cite[Theorem 5]{GS64}
Let $G\colon \R^{m \times n} \to [0, \infty[$ be a~convex function of linear growth and $w \in BV(\Omega, \R^n)$. Then, 
	\begin{equation*}
	G(w_x)(\Omega)
	= \inf \left \{ \liminf_{\ell \to \infty} \int_\Omega G(w_x^\ell) \dd \Lb^m\colon w^\ell \in C^1(\Omega), \, w^\ell \to w \text{ in } L^1_{\mathrm{loc}}(\Omega) \right \}.
	\end{equation*}
\end{thm}

Lastly, let us quote the following fact from measure theory.
\begin{lemma} \label{lem: mu leq nu}
	Let $\mu, \nu$ be positive Radon measures on an open set $U \subset \R^m$ which satisfy 
\[
	\int_U \zeta^2 \dd \mu \leq \int_U \zeta^2 \dd \nu
\]	
for any $\zeta \in C_c^1(U)$. Then, $\mu(E) \leq \nu(E) \text{ for any Borel set } E \subset U.$
\end{lemma}

\begin{proof}
	Given any open set $V \subset U$, it is possible to find an increasing family $\{K_i\}_i$ of compact sets contained in $V$ whose union equals $V$. For each $i$ choose a~non-negative function $\zeta_i \in C_c^1(V)$ with $||\zeta_i||_\infty \leq 1$ which equals $1$ on $K_i$. Consequently, the assumption on $\mu$ and $\nu$ yields
\begin{equation*}
\mu(K_i) \leq \int_U \zeta_i^2 \dd \mu \leq \int_V \zeta_i^2 \dd \nu \leq  \nu(V),
\end{equation*}
which by continuity of measure $\mu$ means that $\mu(V) \leq \nu(V)$ for any open $V \subset U$. Now, if $E$ is an arbitrary Borel set, then
\begin{equation*}
	\mu(E) = \inf{ \left \{\mu(V)\colon V \text{ open, } E \subset V \right \}} \leq \inf{ \left\{ \nu(V):\, V \text{ open, } E \subset V \right \}} = \nu(E).
\end{equation*}
\end{proof}

\subsection*{Convex functions}

\begin{defn} \label{def: Reshetnyak}
Let $\psi\colon \R^n \to [0, \infty[$ be a~positively $1$-homogeneous function. It is convex if and only if for any two points $p, q \in \R^n$ it is true that
\begin{equation} \label{ineq: convexity}
	\psi(p + q) \leq \psi(p) + \psi(q).
\end{equation}
Following Reshetnyak's seminal paper \cite{resh}, we will say that a~positively $1$-homogeneous function $\psi$ is \emph{strictly convex in the sense of Reshetnyak} if and only if inequality \eqref{ineq: convexity} becomes an equality only if $p = 0$ or $q = tp$ for some $t \geq 0$.
\end{defn}

\begin{lemma} \label{lem: strict convexity}
Let $\psi\colon \R^n \to [0, \infty[$ be a positively $1$-homogeneous function which vanishes only at the origin and $g\colon\R \to \R$ be strictly convex and increasing. If $\psi$ is strictly convex in the sense of Reshetnyak, then $G\! := g \circ \psi$ is strictly convex, i.\,e., 
\[G(\lambda p + (1 - \lambda)q) < \lambda G(p) + (1 - \lambda)G(q) \quad \text{for any } \lambda \in ]0,1[ \text{ and distinct } p, q \in \R^n.\] 
If $\psi$ is not strictly convex in the sense of Reshetnyak, then $G$ is not strictly convex.
\end{lemma}

\begin{proof}
Choose any distinct $p, q \in \R^n$ and $0 < \lambda < 1$ and assume strict convexity (in the sense of Reshetnyak) of $\psi$. If $p = 0$, then the statement holds since $g$ is strictly convex and $\psi$ vanishes at the origin (and only there). If $q = tp$ for $t > 0$ and $t \neq 1$, then $\psi(q) = t\psi(p) \neq \psi(p)$ and again it suffices to use strict convexity of $g$. At last, for points for which inequality \eqref{ineq: convexity} is strict, one can write
\begin{align*}
	G\left(\lambda p + (1-\lambda)q \right) &< g \left( \lambda \psi(p) + (1- \lambda) \psi(q) \right) \\
	& \leq \lambda g(\psi(p)) + (1-\lambda)g(\psi(q)).
\end{align*}
The first inequality is strict as $g$ under our assumptions is strictly increasing whereas the second is not as it might happen that $\psi(p) = \psi(q)$.

Suppose now that $\psi$ is not strictly convex (in the sense of Reshetnyak) so that there exist non-zero points $p, q$ such that $p \neq tq$ for any $t > 0$ for which inequality \eqref{ineq: convexity} yields equality. The same is true for $p/2$ and $q/2$, which implies that $\psi$ has to be affine on the segment $[p, q]$. As a~result, for $\lambda = (1 + \psi(p)/\psi(q))^{-1}$ it is true that
\begin{equation*}
	\psi \left( p + \frac{\psi(p)}{\psi(q)}q \right) = \lambda^{-1} \psi( \lambda p + (1 - \lambda)q) = \psi(p) + \psi \left( \frac{\psi(p)}{\psi(q)}q \right).
\end{equation*}
Therefore, we have found points $\tilde{p} = p$, $\tilde{q} = (\psi(p)/\psi(q)) q$ such that $\tilde{p} \neq t\tilde{q}$, whose $\psi$-norms coincide and for which inequality \eqref{ineq: convexity} turns into equality. Consequently, for any $\lambda \in (0,1)$
\begin{equation*}
	G\left(\lambda \tilde{p} + (1-\lambda)\tilde{q} \right)  = g \left( \lambda \psi(\tilde{p}) + (1- \lambda) \psi(\tilde{q}) \right) = g(\psi(\tilde{p})) = \lambda G(\tilde{p}) + (1- \lambda) G(\tilde{q}),
\end{equation*}
what finishes the proof.
\end{proof}

\begin{remark} \label{remark: f(t)/t}
	Given a~convex and differentiable function $f\colon [0, \infty[ \to [0, \infty[$ of linear growth (with constant $C_f^+$) it is true that
	\begin{equation*}
		f^\infty\! := \lim_{t \to \infty} \frac{f(t)}{t} = \lim_{t \to \infty} f'(t) \leq C_f^+.
	\end{equation*}
Observe that in this case, the recession function defined in \eqref{Fmu} equals $f^\infty(s) = f^\infty \, s$. Moreover, convolution of such $f$ with a~standard mollifier yields a~function which is of linear growth with the same constant. This inequality also shows that such $f$ is Lipschitz, so mollifications of $f$ converge uniformly to $f$ (and not only locally). For any convex $\psi\colon \R^n \to \R$, its mollification $\psi_\eta$ with a~radially symmetric mollifier gives $\psi_\eta \geq \psi$.
\end{remark}

Given $f\colon [0, \infty[ \to [0, \infty[$, set
\begin{equation} \label{eq: f eta}
	f_\eta(t) = \left( f(t) \indi_{\{t>\eta\}} + f(\eta)\indi_{\{-\eta \leq t \leq \eta\}} \right) \ast \varrho_{\frac{\eta}{2}}.
\end{equation}
To construct a~suitable sequence approximating a~given anisotropy $\varphi$ we use a~slightly modified mollification procedure which retains homogeneity and which has been described by Schneider in his monograph on convex bodies \cite{schneider}.

\begin{lemma} \label{lem: schneider} \cite[Theorem 3.3.1]{schneider}
	Take $\varrho(p)\! := \rho(|p|)$ with $\rho\colon [0, \infty[ \to [0, \infty[$ being a~smooth function with support contained in $[1/2,1]$ such that $\int_{\R^n} \rho(|z|) \dd z = 1$. For any anisotropy $\varphi$, the function
	\begin{equation}\label{goodnorm_eq}
		\overline{\varphi}_\eta(p)\! := \int_{\R^n} \varphi(p + |p|z) \varrho_\eta(z) \dd z
	\end{equation}
	is an anisotropy of class $C^{\infty}(\R^n \setminus \{0\})$.
\end{lemma}

\begin{prop}\label{goodnorm}
	Let $\varphi$ be an anisotropy on $\R^n$. There exists a~sequence of anisotropies $\{ \ourphi \}_\eta$ with the following properties:
	\begin{enumerate}[(i)]
		\item $\ourphi \in C^\infty(\R^n \setminus \{0\})$ and $D^2 \left( \frac{1}{2} \ourphi^{\,2} \right) \geq \eta \Id$ in $\R^n \setminus \{0\}$;
		\item $\ourphi \geq \varphi$ and $\ourphi$ converge to $\varphi$ locally uniformly on $\R^n$.
	\end{enumerate}
	Assume additionally that $\varphi \in C^{1}(\R^n \setminus \{0\})$. Then,
	\begin{enumerate}[(i)]
		\setcounter{enumi}{2}
		\item $D \ourphi$ converges to $D \varphi$ uniformly on $\R^n \setminus \{0\}$.
    \end{enumerate}
    If $f\colon [0, \infty[ \to [0, \infty[$ is strictly convex, increasing and differentiable with $f'(0) = 0$, then
    \begin{enumerate}[(i)]
        \setcounter{enumi}{3}
		\item $D(f_\eta \circ \ourphi)$ converges to $D(f \circ \varphi)$ uniformly on $\R^n$, where $f_\eta$ is defined in \eqref{eq: f eta}.
	\end{enumerate}	
\end{prop}
Anisotropies which satisfy condition (i) are often referred to as anisotropies of class $C^\infty_+$ (see \cite{schneider}, \cite{BellettiniCasellesNovaga2005}).

\begin{proof}
	We define
	\[\ourphi(p) = \sqrt{\overline{\varphi}_\eta^2(p) + \eta |p|^2}, \] 
	where $\overline{\varphi}_\eta$ is given by \eqref{goodnorm_eq}. Condition (i) is clearly satisfied whereas $\ourphi \geq \varphi$ follows from the fact that $\overline{\varphi}_\eta \geq \varphi$. Indeed, due to homogeneity, it suffices to check it only for $p$ from the unit sphere, for which \eqref{goodnorm_eq} amounts to the ordinary convolution for which this property can be easily shown to be true. Local uniform convergence of $\ourphi$ to $\varphi$ holds since $\overline{\varphi}_\eta$ converges to $\varphi$ locally uniformly as well.
	
	The rest of the proof is devoted to the regular case. The key to uniform convergence of derivatives is the fact that the derivative of a~positively $1$-homogeneous function is $0$-homogeneous, hence it is enough to check it over the unit sphere. Passing with the differentiation under the integral sign gives for $|p| = 1$
	\begin{equation*}
		D \overline{\varphi}_\eta (p) = \int_{\R^n} D\varphi(p + z) \left( \mathrm{Id} + z \cdot p \otimes p \right) \varrho_\eta(|z|) \dd z.
	\end{equation*}
	Uniform continuity of $D\varphi$ on compact sets implies that $D \overline{\varphi}_\eta$ converges uniformly to $D\varphi$. Then, simple calculations lead to the conclusion that property (iii) is satisfied.
	
	The proof of (iv) starts with showing that $f' \circ \ourphi$ converges to $f'\ \circ \varphi$ uniformly on $\R^n$. Recall that for any $\delta$ there exists an $M$ such that if only $t \geq M$, then $f^\infty - f'(t) \leq \delta$. Fix any $\delta > 0$ and choose $M$ accordingly. On the compact ball $\{p\colon \varphi(p) \leq M\}$ functions $\ourphi$ converge uniformly to $\varphi$, i.\,e., there exists $\eta_0$ such that for all $\eta \leq \eta_0$ and $p$ with $\varphi(p) \leq M$ we have
	\[
	|\ourphi(p) - \varphi(p)| \leq \delta.
	\]
	Whenever $\varphi(p) \geq M$, then by (ii) $\ourphi(p) \geq M$ and as a~result, after adding and subtracting $f^\infty$, we get
	\begin{equation*}
		| f'(\ourphi(p)) - f'(\varphi(p)) | \leq 2\delta,
	\end{equation*}
	which shows the desired claim.
	
	Clearly, $D(f_\eta \circ \ourphi)(0) = D(f \circ \varphi)(0) = 0$. Take any $p \neq 0$ and observe that
	\begin{equation*}
		|D(f_\eta \circ \ourphi)(p) - D(f \circ \varphi)(p)| \leq f^\infty |D\ourphi(p) - D\varphi(p)| + |f'_\eta(\ourphi(p)) - f'(\varphi(p))||D\varphi(p)|.
	\end{equation*} 
	By (iii), it is evident that the first term converges uniformly to zero. Moreover, $|D\varphi|$ is bounded. Adding and subtracting $f'(\ourphi(p))$ in the second term together with the claim established in the previous paragraph and the fact that $f'_\eta$ converge uniformly to $f'$ allows us to conclude that the second term converges uniformly to zero as well.
\end{proof}

\begin{lemma} \label{lem: smooth phi sqr}
	There exists a convex function $\widetilde{\widetilde{\varphi}}_\eta \in C^2(\R^n)$, which equals $\ourphi^2$ outside $B_\eta(0)$.
\end{lemma}

\begin{proof}
	To construct such a~function, mollify $\ourphi^2$ with $\varrho_{\eps}$ for a~sufficiently small $\eps$ and glue it to $\ourphi^2$ along a~small annulus in $B_\eta(0)$. Convexity is achieved thanks to the fact that $D^2\ourphi^2 \geq \eta \Id$ on the boundary of the ball $B_\eta(0)$ and the appropriate choice of $\eps$, as in \cite[Section 2]{ghomi}.
\end{proof}

\subsection*{Spherical compactification}
\begin{defn}
Let $\overline{\R^n} := \R^n \cup \mathbb{S}^{n-1}_\infty$, where $\mathbb{S}^{n-1}_\infty\! := \{ \infty \, \omega\colon \omega \in \mathbb{S}^{n-1} \}$. We call $\overline{\R^n}$ the spherical compactification of $\R^n$. Furthermore, define $\Phi\colon \overline{\R^n} \to \overline{\mathbb{B}^n}$ with the formula
\begin{equation*}
	\Phi(p) = \begin{cases}
		\frac{p}{1 + |p|} & \text{if } p \in \R^n, \\
		\omega & \text{if } p = \infty \omega \in \mathbb{S}^{n-1}_\infty.
	\end{cases}
\end{equation*}
The space $\overline{\R^n}$ can be equipped with the metric
\begin{equation*}
	\overline{d}(p, q)\! := \lvert \Phi(p) - \Phi(q) \rvert.
\end{equation*}
\end{defn}

The space $\overline{\R^n}$ with topology induced by $\overline{d}$ is indeed compact since $\Phi$ is a~homeomorphism between this space and the closed unit ball with the Euclidean subspace topology. Clearly, any point $p$ in $\R^n$ can be uniquely represented as $q = r \omega$ with $\omega \in \mathbb{S}^{n-1}_\infty$. If one looks at a~sequence of points $p_j = r_j \omega_j$, it is evident that it converges to a~point $p = \infty \omega \in \mathbb{S}^{n-1}_\infty$ if and only if $r_j \to +\infty$ and $\omega_j \to \omega$. On the other hand, if a~sequence of points $p_j$ converges in $\overline{d}$ to a~point $p \in \R^n$, then for all but a~finite number of $j$, $p_j$ lies in $\R^n$ as well. Moreover, uniform continuity of $\Phi^{-1}$ implies that in such case $|p_j - p| \to 0$. Additionally, since $\Phi$ is $1$-Lipschitz, a~sequence of points $p_j \in \R^n$ which converges in Euclidean distance to $p \in \R^n$, is also convergent to $p$ in $\overline{\R^n}$ with metric $\overline{d}$.

\begin{lemma} \label{lem: injective on spheres}
Let $\varphi\colon \R^n \to [0, \infty[$ be a~$C^1(\R^n \setminus \{0\})$ strictly convex (in the sense of Reshetnyak) anisotropy. Then for any $r>0$ the mapping $D\varphi$ restricted to $\{ p\in \R^n\colon \varphi(p) = r\}$ is injective.
\end{lemma}

\begin{proof}
The function $\varphi^2$ is $C^1$ on the whole $\R^n$ and is strictly convex as shown in Lemma \ref{lem: strict convexity}. Therefore, $D\varphi^2$ is strictly monotone and hence injective. For any $p \neq q$ such that $\varphi(p) = \varphi(q) = r$ for some positive $r > 0$ it is then true that
\[
	2 r D\varphi(p) \neq 2 r D\varphi(q),
\]
which means that $D\varphi(p) \neq D \varphi(q)$.
\end{proof}

\begin{lemma} \label{lem: DF homeo}
Let $F = f \circ \varphi$ be as in Theorem \ref{thm:regular norm}. The mapping $\overline{DF}\colon \overline{\R^n} \to \R^n$ defined as
\begin{equation*}
	\overline{DF} = \begin{cases}
		DF(p) &  \text{if } p \in \R^n,  \\
		f^\infty D\varphi(\omega) & \text{if } p = \infty \omega \in \mathbb{S}^{n-1}_\infty
	\end{cases}
\end{equation*}
is a homeomorphism onto its image.
\end{lemma}

\begin{proof}
Since $\overline{\R^n}$ is compact, it suffices to prove that $\overline{DF}$ is continuous and injective. Continuity is easily established by considering three cases. Firstly, take a~point $p_0 \in \R^n$ and any sequence $p_j$ converging to $p_0$ in metric $\overline{d}$. Then, necessarily, almost all $p_j$ lie in $\R^n$ as well, so continuity at such a~point follows from continuity of $DF$. Secondly, if one approaches $p_0 = \infty \omega_0 \in \mathbb{S}_\infty^{n-1}$ with a~sequence of points $(p_j) \subset \R^n$, then writing $p_j = r_j \omega_j$ with $r_j \geq 0$, $\omega_j \in \mathbb{S}^{n-1}$ gives
\[
	DF(p_j) = f'(r_j \varphi(\omega_j))D\varphi(\omega_j) \to f^\infty D\varphi(\omega_0) = \overline{DF}(p_0),
\]
as $\{\varphi(\omega_j)\}_j$ is bounded away from $0$ and $r_j \to +\infty$, and $D\varphi$ is continuous on the unit sphere. Lastly, continuity of $\overline{DF}$ restricted to $\mathbb{S}^{n-1}_\infty$ follows directly from continuity of $D\varphi$ on the unit sphere.

To prove injectivity, observe that $F$ being strictly convex implies $DF\colon \R^n \to \R^n$ being strictly monotone, so $DF$ is one-to-one. Therefore, $\overline{DF}(p) \neq \overline{DF}(q)$ for any distinct $p, q \in \R^n$. In view of Lemma \ref{lem: injective on spheres}, it suffices to check that $\overline{DF}(p) \neq \overline{DF}(q)$ for $p \in \R^n$ and $q \in \mathbb{S}_\infty^{n-1}$. To this end, let us recall that for any $p \in \R^{n}\setminus \{0\}$ it is true that
\begin{equation} \label{dual norm derivative}
	\varphi^* \left(D\varphi(p) \right) = 1, \text{ where } \varphi^*(p)\! : = \sup \left \{ p \cdot q\colon  \varphi(q) \leq 1 \right \}.
\end{equation}
Indeed, taking $q = p/\varphi(p)$ gives $\varphi^*\left(D\varphi(p)\right) \geq 1$. On the other hand, the product $q \cdot D\varphi(p)$ equals the derivative of $\varphi$ at point $p$ in direction $q$ and hence, by triangle inequality,
\begin{equation*}
	\varphi^*\left( D\varphi(p) \right) = \sup_{\varphi(q) \leq 1} \lim_{t \to 0^+} \frac{1}{t} \left( \varphi(p + tq) - \varphi(p) \right) \leq \sup_{\varphi(q) \leq 1} \varphi(q) = 1.
\end{equation*}

Assume now that $DF(p) = f^\infty D\varphi(\omega)$ for some $p \in \R^n \setminus \{0\}$ and $\omega \in \mathbb{S}^{n-1}_1$. Applying $\varphi^*$ to both sides and using its positive homogeneity results in
\begin{equation*}
	f'(\varphi(p)) \varphi^* \left( D \varphi(p) \right) = f^\infty \varphi^*\left( D\varphi(\omega) \right).
\end{equation*}
Given \eqref{dual norm derivative}, one gets $f'(\varphi(p)) = f^\infty$, which contradicts strict convexity of $f$. Therefore, $\overline{DF}$ is injective.
\end{proof}

\section{Approximate functionals} \label{sec: approximate}

Suppose $h \in L^2(I, \R^n) \cap BV(U, \R^n)$ for an open interval $U \subset I$ and extend $h$ by zero beyond $I$ in order to properly define mollifications of $h$. Set
\begin{equation}
	h^\eps = h \ast \varrho_\eps.
\end{equation}
We suppose that $F_{\eta}$ is given either by 
\begin{equation} \label{Feta1def}
	F_\eta = f_\eta \circ \overline{\varphi}_\eta + \frac{\eta}{2}|\cdot|^2 
\end{equation} 
or 
\begin{equation} \label{Feta2def}
	F_\eta = f_\eta \circ \ourphi + \frac{\eta}{2} \widetilde{\widetilde{\varphi}}_\eta.
\end{equation}
In \eqref{Feta1def}, \eqref{Feta2def} we have used $f_\eta$, $\overline{\varphi}_\eta$, $\widetilde{\varphi}_\eta$ and $\widetilde{\widetilde{\varphi}}_\eta$ defined in \eqref{eq: f eta}, Lemma \ref{lem: schneider}, Proposition \ref{goodnorm} and Lemma \ref{lem: smooth phi sqr} for given anisotropy $\varphi \colon \R^n \to [0, \infty[$ and convex, increasing $f \colon [0, \infty[ \to [0, \infty[$ of linear growth. Observe that $F_\eta$ is smooth and that $D^2 F_\eta \geq \eta \Id$. Moreover, since $\ourphi \geq \varphi$ and $f_\eta \geq f$ due to convexity of $f$, it is true that $F_\eta \geq F$.

We will use a two-layer approximation of $\E$: 
\begin{align*}
	\E_{\eps, \eta} (w) &= \lambda \int_I F_\eta(w_x) \dd \Lb^1 + \frac{1}{2} \int_I |w - h^{\eps} |^2 \dd \Lb^1 \text{ for } w \in W^{1,2}(I, \R^n) \text{ with mimimizer } u^{\eps, \eta}, \\
	\E_\eps(w) &= \lambda \int_I F(w_x) \dd \Lb^1  + \frac{1}{2} \int_I |w-h^\eps|^2 \dd \Lb^1 \text{ for } w \in L^2(I, \R^n) \text{ with mimimizer } u^{\eps}. 
\end{align*}
Observe that $F_\eta$ and $F$ are strictly convex, non-negative, weakly lower semicontinuous in $W^{1,2}$ and $L^2$, respectively. Moreover the sublevels of $\E_{\eps, \eta}$, $\E_\eps$ are weakly compact in $W^{1,2}$ and $L^2$, respectively. Consequently, there indeed exist unique minimizers of these functionals.

\begin{lemma} \label{lem: gamma conv eta}
	Functionals $\E_{\eps, \eta}$ $\Gamma$-converge to $\E_{\eps}$ w.\,r.\,t.\ weak $W^{1,2}(I, \R^n)$ convergence.
\end{lemma}
\begin{proof}
	As $F_\eta \geq F$, then $\E_{\eps, \eta} \geq \E_\eps$, hence lower limit inequality follows from lower semicontinuity of $E_\eps$, see e.\,g.\ \cite[Chapter 8.2, Theorem 1]{evans}.
	As for the recovery sequence, choose $w_{\eta} = w$, then one only needs to show that
	\begin{equation} \label{f eta limit}
		\lim_{\eta \to 0}\int_I F_\eta(w_x) \dd \Lb^1 = \int_I F(w_x) \dd \Lb^1.
	\end{equation}
	Since $\overline{\varphi}$ (resp.\ $\ourphi$) is an anisotropy and $f_\eta$ is of linear growth, we have $f_\eta \circ \overline{\varphi} (w_x) \leq C( 1 + |w_x|)$ (resp.\ $f_\eta \circ \ourphi (w_x) \leq C( 1 + |w_x|)$), which is integrable over $I$, implying that dominated convergence theorem can be applied. Additionally, we have
	\begin{equation*}
		\lim_{\eta \to 0} \frac{\eta}{2} \int_I \widetilde{\widetilde{\varphi}}_\eta (w_x) \dd \Lb^1 \leq \lim_{\eta \to 0} \frac{\eta}{2} \int_I C (1 + |w_x|^2) \dd \Lb^1 = 0,
	\end{equation*}
	which proves \eqref{f eta limit} and concludes the proof of the lemma.
\end{proof}

\begin{lemma} \label{lem: gamma conv eps 1}
	Functionals $\E_\eps$ $\Gamma$-converge to $\E$ w.\,r.\,t.\ weak convergence in $L^2(I, \R^n)$.
\end{lemma}
\begin{proof}
	Since $h^\eps$ converge strongly to $h$ in $L^2(I, \R^n)$, for each $w \in BV(I, \R^n)$ there exists a~trivial recovery sequence. 
	
	Then, take $w \in BV(I, \R^n)$ and a~sequence $w^\eps$ converging to $w$ weakly in $L^2$. Since the norm is lower semicontinuous w.\,r.\,t.\ weak convergence,
	\begin{equation*}
		\liminf_{\eps \to 0} \int_I |w^\eps - h^\eps|^2 \dd \Lb^1 \geq \int_I |w - h|^2 \dd \Lb^1. 
	\end{equation*}
	
	Lower semicontinuity of $\F$ with respect to strong $L^2(I, \R^n)$ convergence follows from the definition of $\F$, as $L^2$ convergence is stronger than $L^1_{loc}$ convergence. Since $\F$ is convex, it is also lower semicontinuous with respect to weak $L^2$ convergence.
\end{proof}

Since $\E_{\eps, \eta}$ is smooth, the minimizer admits weak second derivatives and satisfies the following Euler--Lagrange equation
 \begin{equation} \label{Euler-Lagrange-eps-eta}
	u^{\eps, \eta} - h^\eps = \lambda \left(DF_\eta(u_x^{\eps, \eta})\right)_x \text{ in } I, \quad u_x^{\eps, \eta} = 0 \text{ on } \partial I.
\end{equation}
Due to the fact that $D^2F_\eta \geq \eta \Id$, we gather that $u^{\eps, \eta} \in W^{2,2}(I, \R^n)$, which in particular means that $u^{\eps, \eta} \in C^{1}(\bar{I}, \R^n)$.

\begin{lemma} \label{lem: cpt eta}
As $\eta \to 0$, the sequence of minimizers $u^{\eps, \eta}$ converges strongly in $L^2(I, \R^n)$ and weakly in $W^{1,2}(I, \R^n)$ to the minimizer $u^\eps$.
\end{lemma}
\begin{proof}
Firstly, let us prove that the sequence of minimizers $u^{\eps, \eta}$ is bounded in $W^{1,2}$. The fact that $\E_{\eps, \eta}(u^{\eps, \eta}) \leq \E_{\eps, \eta}(0)$ together with triangle inequality for $L^2$ norm shows that $u^{\eps, \eta}$ are bounded in $L^2$. For the boundedness of the derivatives, let us multiply the Euler--Lagrange equation by $u_{xx}^{\eps, \eta}$ and integrate the l.\,h.\,s.\ by parts. The r.\,h.\,s\ is non-negative due to convexity of $F_\eta$ and the boundary term vanishes, hence
	\begin{equation*}
		\int_{I} |u_x^{\eps, \eta}|^2 \dd \Lb^1 \leq \int_I u_x^{\eps, \eta} \cdot h_x^\eps \dd \Lb^1.
	\end{equation*}
Applying Young's inequality gives the desired bound for $L^2$ norm of the sequence of derivatives of minimizers.

Therefore, from any subsequence $u^{\eps, \eta}$ it is possible to choose a~weakly convergent subsequence $u^{\eps, \eta_j}$. In view of the $\Gamma$-convergence proved in Lemma \ref{lem: gamma conv eta}, the weak limit of $\{u^{\eps, \eta_j}\}_j$ must equal the minimizer $u^\eps$. A standard application of Rellich--Kondrashov theorem allows us to choose a~further subsequence (which we do not relabel) which additionally converges strongly in $L^2$ to $u^\eps$. Therefore, a~properly convergent subsequence can be chosen from arbitrary subsequence of $\{u_x^\eps\}_\eps$, which shows the desired convergence of the whole sequence.
\end{proof}

\begin{lemma} \label{lem: gamma conv eps} 
As $\eps \to 0$, the minimizers $u^\eps$ converge strongly in $L^2(I, \R^n)$ to the minimizer $u$ of $\E$ and $u_x^\eps \weaklystar u_x$ in $M(I, \R^n)$.
\end{lemma}
\begin{proof}
    To see that $u^\eps$ are bounded in $L^2$, it suffices to write $\E_\eps(u^\eps) \leq \E_\eps(0)$ and use triangle inequality for $L^2$ norm. Consequently, from every sequence $u^{\eps_i}$ it is possible to choose a~further subsequence which, by $\Gamma$-convergence proved in Lemma \ref{lem: gamma conv eps 1}, converges weakly in $L^2$ to $u$.

    Moreover, using the lower bound \ref{lin_growth} for $F$, one can see that $\{u^{\eps_i}\}_i$ is bounded in $BV$, which implies an existence of a~subsequence strongly convergent in $L^2$ (see \cite[Corollary 3.49]{ambrosiofuscopallara}) such that the derivatives converge weakly* as measures. 
\end{proof}

\section{The homogeneous case} \label{sec:homo}

\begin{proof}[Proof of Theorem \ref{thm:homo}]
In this section we choose $F_\eta$ defined as in \eqref{Feta1def}. We calculate 
\[
DF_{\eta}(p) = f_\eta'(\overline{\varphi}_\eta(p)) D \overline{\varphi}_\eta(p) + \eta p, 
\]
\[
\left(DF_\eta(u_x^{\eps, \eta})\right)_x = f_\eta'(\overline{\varphi}_\eta(u_x^{\eps, \eta})) D^2 \overline{\varphi}_\eta(u_x^{\eps, \eta}) u_{xx}^{\eps, \eta} + f_\eta''(\overline{\varphi}_\eta(u_x^{\eps, \eta})) \overline{\varphi}_\eta(u_x^{\eps, \eta})_x D \overline{\varphi}_\eta(u_x^{\eps, \eta}) + \eta\, u_{xx}^{\eps, \eta}. 
\]
Testing the Euler--Lagrange equation \eqref{Euler-Lagrange-eps-eta} with $u^{\eps, \eta}_{xx}$, integrating by parts on the l.\,h.\,s.\ and using Cauchy's inequality we obtain a uniform estimate
\begin{equation} \label{homo_eta_unif_bd} 
\int_I |u_x^{\eps, \eta}|^2 \dd \Lb^1 + 2 \lambda \eta \int_I |u_{xx}^{\eps, \eta}|^2 \dd\Lb^1 \leq \int_I |h_x^{\eps}|^2 \dd \Lb^1 .
\end{equation} 

We define 
\begin{equation} 
	\mathfrak G_k (p) = (|p|-k)_+, \quad   \mathfrak G_{k, \delta} (p) = \sqrt{(|p|-k)_+^2 + \delta^2}.
\end{equation}  
Then, 
\[
D \mathfrak G_{k, \delta} (p) = \frac{(|p|-k)_+}{\sqrt{(|p|-k)_+^2 + \delta^2}} \frac{p}{|p|}
\] 
and, given $\zeta \in C_c^1(U)$ 
\begin{multline*}
	\left(D \mathfrak G_{k, \delta} (u_x^{\eps, \eta}) \zeta^2\right)_x  = \frac{(|u_x^{\eps, \eta}|-k)_+}{\sqrt{(|u_x^{\eps, \eta}|-k)_+^2 + \delta^2}} \frac{1}{|u_x^{\eps, \eta}|}\left(I - \frac{u_x^{\eps, \eta}}{|u_x^{\eps, \eta}|} \otimes \frac{u_x^{\eps, \eta}}{|u_x^{\eps, \eta}|}\right)  u_{xx}^{\eps, \eta} \zeta^2 \\ + \frac{\delta^2 \indi_{|u_x^{\eps, \eta}|>k }}{\sqrt{(|u_x^{\eps, \eta}|-k)_+^2 + \delta^2}^3} \frac{u_x^{\eps, \eta}}{|u_x^{\eps, \eta}|} \otimes \frac{u_x^{\eps, \eta}}{|u_x^{\eps, \eta}|}\, u_{xx}^{\eps, \eta} \zeta^2 + \frac{(|u_x^{\eps, \eta}|-k)_+}{\sqrt{(|u_x^{\eps, \eta}|-k)_+^2 + \delta^2}} \frac{u_x^{\eps, \eta}}{|u_x^{\eps, \eta}|} \, 2 \zeta \zeta_x.
\end{multline*}  
We observe that 
\begin{multline}\label{homo_est1} 
	\int_I \left(D \mathfrak G_{k, \delta} (u_x^{\eps, \eta}) \zeta^2\right)_x \cdot u_{xx}^{\eps, \eta} \dd \Lb^1 \geq \int_I 2 \zeta \zeta_x \frac{(|u_x^{\eps, \eta}|-k)_+}{\sqrt{(|u_x^{\eps, \eta}|-k)_+^2 + \delta^2}} \frac{u_x^{\eps, \eta}}{|u_x^{\eps, \eta}|} \cdot u_{xx}^{\eps, \eta} \dd \Lb^1 \\ 
    \geq - 2 \max |\zeta \zeta_x| |I|^\frac{1}{2}\left(\int_I  |u_{xx}^{\eps, \eta}|^2 \dd \Lb^1\right)^{\frac{1}{2}}. 
\end{multline} 

Assuming that $k > \frac{2}{c_\varphi^-} \eta$, we have $|p| \leq k$ or $\overline{\varphi}_\eta(p) \geq 2 \eta$ for all $p \in \R^n$. In the latter case $f'(\overline{\varphi}_\eta(p)) = 1$ and $f''(\overline{\varphi}_\eta(p)) = 0$. Taking into account \eqref{homo_est1}, \eqref{homo_eta_unif_bd} and the observation that 
\[D^2 \overline{\varphi}_\eta(p) \cdot p = 0 \quad \text{for } p \neq 0,\]
we get  
\begin{multline} \label{homo_est2} 
	\int_I \left(D \mathfrak G_{k, \delta} (u_x^{\eps, \eta}) \zeta^2\right)_x \cdot \left(DF_\eta(u_x^{\eps, \eta})\right)_x \dd \Lb^1 \geq - 2 \eta \max |\zeta \zeta_x| |I|^\frac{1}{2} \left(\int_I  |u_{xx}^{\eps, \eta}|^2 \dd \Lb^1 \right)^{\frac{1}{2}} \\ \geq - \frac{\sqrt{2\eta}}{\sqrt{\lambda}} \max |\zeta \zeta_x| |I|^\frac{1}{2}\left(\int_I |h_x^\eps|^2 \dd \Lb^1\right)^{\frac{1}{2}}.  
\end{multline} 
On the other hand, by convexity of $\mathfrak G_{k, \delta}$, 
\begin{multline} \label{homo_rhs_est}  
\int_I \left(D \mathfrak G_{k, \delta} (u_x^{\eps, \eta}) \zeta^2\right)_x \cdot (u^{\eps, \eta} - h^\eps)\dd \Lb^1  = \int_I \zeta^2 D \mathfrak G_{k, \delta} (u_x^{\eps, \eta}) \cdot (h^\eps_x - u^{\eps, \eta}_x)\dd \Lb^1 \\ \leq \int_I \zeta^2 \mathfrak G_{k, \delta} (h_x^\eps)\dd \Lb^1 - \int_I \zeta^2 \mathfrak G_{k, \delta} (u_x^{\eps, \eta})\dd \Lb^1. 
\end{multline} 
Collecting \eqref{homo_est2} and \eqref{homo_rhs_est} and passing to the limit $\delta \to 0^+$ using dominated convergence, we infer that 
\begin{equation} \label{homo_est3} 
	 \int_I \zeta^2 \mathfrak G_k (u_x^{\eps, \eta}) \dd \Lb^1 \leq \int_I \zeta^2 \mathfrak G_k (h_x^\eps) \dd \Lb^1 + \frac{\sqrt{2\eta}}{\sqrt{\lambda}} \max |\zeta \zeta_x| |I|^\frac{1}{2} \left(\int_I |h_x^\eps|^2 \dd \Lb^1\right)^{\frac{1}{2}}.
\end{equation}
As shown in Lemma \ref{lem: cpt eta}, $u_x^{\eps, \eta}$ converge weakly in $L^2$ to $u_x^{\eps}$ as $\eta \to 0^+$. Therefore, by lower semicontinuity of $w_x \mapsto \int \mathfrak G_k(w_x) \dd \Lb^1$, inequality \eqref{homo_est3} yields
\begin{equation} \label{homo_est4} 
	\int_I \zeta^2 \mathfrak G_k (u_x^\eps) \dd \Lb^1 \leq \int_I \zeta^2 \mathfrak G_k (h_x^\eps) \dd \Lb^1.
\end{equation}

By Lemma~\ref{lem: gamma conv eps}, we know that $u_x^\eps \weaklystar u_x$. As a~result, Lemmata \ref{cor-gs-lsc-zeta} and \ref{lem: G weaklystar} mean that the following inequality holds
\[
    \int_U \zeta^2 \dd \mathfrak{G}_k(u_x) \leq \int_U \zeta^2 \dd \mathfrak G_k(h_x).
\]
Employing Lemma \ref{lem: mu leq nu} again goes to show that
\begin{equation*} 
	\mathfrak{G}_k (u_x) \leq \mathfrak{G}_k (h_x) 
\end{equation*}
as measures or, in other words,
\begin{equation*} 
	\mathfrak |u_x^s|(V) + \int_V (|u_x^{ac}|-k)_+ \dd \Lb^1  \leq |h_x^s|(V) + \int_V (|h_x^{ac}|-k)_+ \dd \Lb^1 
\end{equation*}
for any $V$ Borel. Passing to the limit $k\to 0^+$ we obtain the desired assertion.
\end{proof}

\section{The general case} \label{sec:coercive} 
From now on, we choose $F_\eta$ to be defined as in \eqref{Feta2def}. Let us multiply both sides of the Euler--Lagrange equation \eqref{Euler-Lagrange-eps-eta} by $(\zeta^2 G_k(u_x^{\eps, \eta}))_x$, where $\zeta \in C_c^1(U)$ and
\begin{equation*}
	G_k(p) = g_k \circ \tilde{\varphi}_\eta(p)p \text{ and } g_k(\sigma) = \frac{(\sigma - k)_+}{\sigma^2} f'_\eta(\sigma).
\end{equation*}
Observe that $G_k(p)$ is bounded since $f_\eta$ is of linear growth and $\ourphi$ is a~coercive anisotropy. After integrating and differentiating by parts on the l.\,h.\,s. we get
\begin{equation} \label{eq: after mult}
	\frac{1}{\lambda} \int_I \zeta^2 G_k(u_x^{\eps, \eta}) \cdot  (u_x^{\eps, \eta} - h_x^{\eps}) \dd \Lb^1 = - \int_I (\zeta^2 G_k(u_x^{\eps, \eta}))_x \cdot D^2 F_\eta (u_x^{\eps, \eta}) \cdot u_{xx}^{\eps, \eta} \dd \Lb^1.
\end{equation}

\begin{prop} \label{prop: general lhs eps}
	The l.\,h.\,s.\ of \eqref{eq: after mult} satisfies
\begin{align*}
	\liminf_{\eps \to 0} \liminf_{\eta \to 0} & \int_I \zeta^2 G_k(u_x^{\eps, \eta}) \cdot (u_x^{\eps, \eta} - h_x^{\eps}) \dd \Lb^1 \geq \\
	& \frac{r(k)}{ (c_\varphi^+)^{3}} \int_I \zeta^2 \dd (c_\varphi^- |u_x| - k)_+(c_\varphi^+ |u_x| - k)_+ |u_x|^{-1} - \frac{f^\infty}{ c_\varphi^- c_\varphi^+} \int_I \zeta^2 \dd (c_\varphi^+ |h_x| - k)_+,
\end{align*}
where $\lim_{k \to \infty} r(k) = f^\infty$.
\end{prop}

\begin{proof}
	We will firstly show that for a~fixed $\eps$, it is true that
\begin{align} \label{eq: lhs general eta}
\begin{split}
	\liminf_{\eta \to 0} & \int_I \zeta^2 G_k(u_x^{\eps, \eta}) \cdot (u_x^{\eps, \eta} - h_x^{\eps}) \dd \Lb^1 \geq \\
	& \frac{r(k)}{ (c_\varphi^+)^{3}} \int_I \zeta^2 (c_\varphi^- |u_x^{\eps}| - k)_+(c_\varphi^+ |u_x^{\eps}| - k)_+ |u_x^{\eps}|^{-1} \dd \Lb^1 - \frac{f^\infty}{ c_\varphi^- c_\varphi^+} \int_I \zeta^2 (c_\varphi^+ |h_x^\eps| - k)_+ \dd \Lb^1.
\end{split}
\end{align}
Since $\ourphi(p) \geq \varphi(p)$ and due to local uniform convergence of $\ourphi$ to $\varphi$, one can define a~non-decreasing function $\eta(\delta)$ such that for all $\eta \leq \eta (\delta)$
\begin{equation} \label{eq: norms}
	c_\varphi^- |p| \leq \ourphi(p) \leq (c^+_\varphi + \delta) |p|.
\end{equation}
By inverting it, we get a~non-decreasing function $\delta(\eta)$ that converges to zero as $\eta$ tends to zero and \eqref{eq: norms} holds for $\delta \geq \delta(\eta)$. Let us also note a~consequence of convexity of the function $p \mapsto (C|p| - k)_+$ for $C, k > 0$, namely, the inequality
	\begin{equation} \label{ineq: lhs general}
		C \indi_{C|p| > k} \cdot \frac{p}{|p|} (p - q) \geq (C|p| - k)_+ - (C|q| - k)_+.
	\end{equation}

Using the fact that $(\ourphi(p) - k)_+ \, \indi_{(c^+_\varphi + \delta)|p| > k} = (\ourphi(p) - k)_+$ and \eqref{ineq: lhs general}, we obtain the following estimate
\begin{multline*}
	G_k(u_x^{\eps, \eta}) \cdot (u_x^{\eps, \eta} - h_x^\eps) = \frac{(\ourphi(u_x^{\eps, \eta}) - k)_+}{\ourphi^2(u_x^{\eps, \eta})} f'_\eta(\ourphi(u_x^{\eps, \eta})) u_x^{\eps, \eta} \cdot (u_x^{\eps, \eta} - h_x^\eps) \\
		\overset{\eqref{ineq: lhs general}}{\geq} \frac{(\ourphi(u_x^{\eps, \eta}) - k)_+}{\ourphi^2(u_x^{\eps, \eta})} f'_\eta(\ourphi(u_x^{\eps, \eta})) \frac{|u_x^{\eps, \eta}|}{c_\varphi^+ + \delta}\left [ \left ( (c_\varphi^+ + \delta)|u_x^{\eps, \eta}| - k \right)_+ - \left( (c_\varphi^+ + \delta) |h_x^\eps| - k \right)_+ \right ].
\end{multline*}
Handling these two terms separately and using \eqref{eq: norms} repeatedly yields
\begin{multline} \label{eq: G k ineq}
	G_k(u_x^{\eps, \eta}) \cdot (u_x^{\eps, \eta} - h_x^\eps)  \\
		\geq \frac{f'_\eta(k)}{(c_\varphi^+ + \delta)^3} \underbrace{\left( c_\varphi^-|u_x^{\eps, \eta}| - k \right)_+ \left( c_\varphi^+|u_x^{\eps, \eta}| - k \right)_+ \frac{1}{|u_x^{\eps, \eta}|}}_{A_k(|u_x^{\eps, \eta}|)} - \frac{f'_\eta(\ourphi(u_x^{\eps, \eta}))}{c_\varphi^-(c_\varphi^+ + \delta)}\left( (c_\varphi^+ + \delta)|h_x^\eps| - k \right)_+  \\
		\geq \frac{f(k) - f(0) -1}{k} (c_\varphi^+ + \delta)^{-3} A_k(|u_x^{\eps, \eta}|) - \frac{f^\infty}{c_\varphi^-c_\varphi^+} \left( (c_\varphi^+ + \delta)|h_x^\eps| - k \right)_+.
\end{multline}

The last inequality essentially follows from convexity of $f$. Indeed, due to uniform convergence of $f_\eta$ to~$f$, $||f - f_\eta||_\infty \leq 1$ for sufficiently small $\eta$. Together with convexity of $f_\eta$ and the fact that $f_\eta \geq f$, it means that
\[
	\frac{f(t) - f(0) - 1}{t} \leq \frac{1}{t}(f_\eta(t) - f_\eta(0)) \leq \, f'_\eta(t) \, \leq \frac{1}{t}( f_\eta(2t) - f_\eta(t)) \leq \frac{f(2t) + 1}{t} - \frac{f(t)}{t}.
\] 
Taking limit with $t \to \infty$ yields $\lim_{t \to \infty} f'_\eta(t) \leq f^\infty$, which was used in the estimate in the second term, $f'_\eta(\ourphi(u_x)) \leq \lim_{t \to \infty} f'_\eta(t)$. On the other hand, plugging $t = k$ on the l.\,h.\,s.\ of this inequality explains the estimate in the first term. Let $r(k) := (f(k) - f(0) - 1)/k$ and observe that $r(k) > 0 $ for sufficiently large $k$ and that $\lim_{k \to \infty} r(k) = f^\infty$.

At this point, we multiply the inequality \eqref{eq: G k ineq} by $\zeta^2$ and integrate it over $I$. Since $h_x^\eps$ is integrable, it follows by the dominated convergence theorem that
\begin{equation} \label{eq: lim for h}
	\lim_{\eta \to 0} \frac{f^\infty}{c_\varphi^- c_\varphi^+} \int_I \zeta^2 \left( (c_\varphi^+ + \delta(\eta))|h_x^{\eps} | - k \right)_+ \dd \Lb^1 = \frac{f^\infty}{c_\varphi^- c_\varphi^+} \int_I \zeta^2 (c_\varphi^+|h_x^\eps| - k)_+ \dd \Lb^1.
\end{equation}
In the term involving $u_x^{\eps, \eta}$ it will not be possible to take the limit but it will suffice to estimate (recall that Lemma \ref{lem: cpt eta} established that $u^{\eps,\eta}$ converges weakly in $W^{1,2}$ to $u^\eps$)
\begin{align} \label{eq: liminf for u}
\begin{split}
	\liminf_{\eta \to 0} r(k)(c_\varphi^+ + \delta(\eta))^{-3} \int_I \zeta^2 A_k( |u_x^{\eps, \eta}|) \dd \Lb^1 &\geq r(k) (c_\varphi^+)^{-3} \, \liminf_{\eta \to 0} \int_I \zeta^2 A_k(|u_x^{\eps, \eta}|) \dd \Lb^1 \\
	&\geq r(k) (c_\varphi^+)^{-3} \int_I \zeta^2 A_k(|u_x^{\eps}|) \dd \Lb^1.
\end{split}
\end{align}
The last inequality follows from convexity of $p \mapsto A_k(|p|)$ (it is easy to check that $A_k(s)$ is convex and increasing). Indeed, since $u_x^{\eps, \eta}$ converge weakly in $L^2$ to $u_x^\eps$, lower semicontinuity follows from the standard result \cite[Chapter 8.2, Theorem 1]{evans}. Thus, we arrive at the desired inequality \eqref{eq: lhs general eta}.

We now turn to the passage with $\eps \to 0$. Plugging $G(p) = (c_\varphi^+|p| - k)_+$ in Lemma \ref{lem: G weaklystar} yields
	\[
		\lim_{\eps \to 0} \int_I \zeta^2(c_\varphi^+|h_x^\eps| - k)_+ \dd \Lb^1 = \int_I \zeta^2 \dd (c_\varphi^+|h_x| - k)_+.
	\]
As shown in Lemma \ref{lem: gamma conv eps}, $u_x^\eps \weaklystar u_x$. After applying Corollary \ref{cor-gs-lsc-zeta}, we obtain
\[
\liminf_{\eps \to 0} \int_I \zeta^2 A_k(|u_x^{\eps}|) \dd \Lb^1 \geq \int_I \zeta^2 \dd A_k(|u_x|),
\]
which finishes the proof.
\end{proof}

\begin{prop} \label{prop-rhs-eps-eta}
Let $\mathcal{R}(\eps, k)$ denote the upper limit of the r.\,h.\,s.\ of \eqref{eq: after mult} as $\eta \to 0$, i.\,e.,
\begin{equation}\label{def_R} 
\mathcal{R}(\eps, k) := \limsup_{\eta \to 0} - \int_I (\zeta^2 G_k(u_x^{\eps, \eta}))_x \cdot D F_\eta (u_x^{\eps, \eta})_x \dd \Lb^1
\end{equation} 
Then
\[
	\lim_{k \to \infty} \limsup_{\eps \to 0} \mathcal{R}(\eps, k) = 0.
\]
\end{prop}

\begin{proof}
Throughout the proof, the superscript $\eps, \eta$ on $u_x^{\eps, \eta}$ will be omitted for readability. We compute
\begin{equation}\label{comp_gen}
\begin{aligned}
	(\zeta^2 G_k(u_x))_x &= 2 \zeta \zeta_x g_k(\ourphi(u_x)) u_x + \zeta^2 g_k'(\ourphi(u_x))(\ourphi(u_x))_x u_x + \zeta^2 g_k(\ourphi(u_x)) u_{xx},\\
    D F_\eta(u_x)_x &= (D (f_\eta \circ \ourphi)(u_x))_x + \frac{\eta}{2} D \widetilde{\widetilde{\varphi}}_\eta(u_x)_x.  
\end{aligned}
\end{equation}
We deal with the terms coming from $f_\eta \circ \ourphi$ and $\widetilde{\widetilde{\varphi}}_\eta$ separately. First, we calculate
\begin{equation*}
    (D (f_\eta \circ \ourphi)(u_x))_x =  f''_\eta(\ourphi(u_x)) \ourphi(u_x)_x D\ourphi(u_x) + f'_\eta(\ourphi(u_x)) D^2 \ourphi(u_x) \cdot u_{xx}, 
\end{equation*}
\begin{equation} \label{rhs_est_gen}
\begin{aligned}
		(\zeta^2 G_k(u_x))_x \cdot (D(f_\eta \circ \ourphi)(u_x))_x&=
	2\zeta \zeta_x g_k(\ourphi(u_x)) f''_\eta(\ourphi(u_x)) \ourphi(u_x) (\ourphi(u_x))_x \\
	&+ \zeta^2 f''_\eta(\ourphi(u_x)) (\ourphi(u_x))_x^2  \left( g_k'(\ourphi(u_x)) \ourphi(u_x)  + g_k(\ourphi(u_x)) \right)\\
	&+ \zeta^2 g_k(\ourphi(u_x)) f'_\eta(\ourphi(u_x)) u_{xx} \cdot D^2 \ourphi(u_x) \cdot u_{xx}.
\end{aligned} 
\end{equation}
Other terms vanish because $D^2 \ourphi(u_x) \cdot u_x = 0$ by homogeneity of $\ourphi$. The last term on the r.\,h.\,s.\ of \eqref{rhs_est_gen} is non-negative and we will ignore it. There holds 
\[g_k(\sigma) + g'_k(\sigma)\sigma = \sigma^{-2}\left(k f'_{\eta}(\sigma) + \sigma(\sigma - k) f''_{\eta}(\sigma)\right) \mathbb{1}_{\sigma > k} \geq 0.\]
Thus, the second term on the r.\,h.\,s.\ of \eqref{rhs_est_gen} is also non-negative and we can use to estimate the first term by means of inequality $2ab \geq -a^2 - b^2$ with suitable choices of $a, b$, obtaining
\begin{multline*} 
		(\zeta^2 G_k(u_x))_x \cdot D^2(f_\eta \circ \ourphi )(u_x)\cdot u_{xx} \\ \geq -\zeta_x^2 f''_\eta(\ourphi(u_x)) \frac{g_k^2(\ourphi(u_x)) \ourphi^2(u_x)}{g_k(\ourphi(u_x)) + g'_k(\ourphi(u_x))\ourphi(u_x)} = -\zeta_x^2 f''_\eta(\ourphi(u_x)) \mathcal{K}(\ourphi(u_x)),
\end{multline*}
where we have denoted \begin{equation} \label{K_def}
\mathcal{K}(\sigma) = \frac{g_k^2(\sigma) \sigma^2}{g_k(\sigma) + g'_k(\sigma)\sigma}= \frac{(\sigma - k)_+^2  | f'_{\eta} (\sigma) |^2}{k f'_{\eta}(\sigma) + \sigma(\sigma - k) f''_{\eta}(\sigma)}.
\end{equation}
Forgetting about the non-negative term $k f'_\eta$ in the denominator and using linear growth of $f_\eta$ allows us to conclude that
\begin{equation} \label{ke-small}
\zeta^2_x f''_\eta (\sigma) \mathcal{K}(\sigma) \leq \frac{||\zeta_x||_\infty^2}{\sigma}(\sigma - k)_+ |f'_\eta(\sigma)|^2 \leq \frac{C ||\zeta_x||_\infty^2}{k} (\sigma - k)_+.
\end{equation}

Let us now turn to the neglected terms of order $\eta/2$. Recall that outside $B_\eta(0)$ there holds $\widetilde{\widetilde{\varphi}}_\eta = \ourphi^2$ and thus
\[\frac{\eta}{2} D \widetilde{\widetilde{\varphi}}_\eta(u_x)_x = \eta \ourphi(u_x)_x D \ourphi(u_x) + \eta \ourphi(u_x) D^2\ourphi(u_x) \cdot u_{xx}.\]
Since $G_k(p)$ vanishes on $B_k(0)$, as long as $\eta < k$ we have, recalling \eqref{comp_gen}, 

\begin{align*}
	\frac{1}{2}(\zeta^2 G_k(u_x))_x \cdot D \widetilde{\widetilde{\varphi}}_\eta(u_x)_x &=
	2\zeta\zeta_x g_k(\ourphi(u_x)) \ourphi(u_x) (\ourphi(u_x))_x \\
	&+ \zeta^2 \ourphi(u_x)_x^2 \left(g_k'(\ourphi(u_x))\ourphi(u_x) + g_k(\ourphi(u_x) \right) \\
	&+ \zeta^2g_k(\ourphi(u_x)) \ourphi(u_x) u_{xx} \cdot D^2\ourphi(u_x) \cdot u_{xx} . 
\end{align*}
Observing that the last two terms on the r.\,h.\,s.\ are non-negative and estimating the first term as before, we obtain 
\begin{equation*} 
\frac{1}{2}(\zeta^2 G_k(u_x))_x \cdot D \widetilde{\widetilde{\varphi}}_\eta(u_x)_x \geq -\zeta_x^2 \mathcal{K}(\ourphi(u_x)),
\end{equation*} 
with $\mathcal{K}$ as in \eqref{K_def}. Forgetting about the term with $f''_\eta$ in the denominator of $\mathcal{K}$ and keeping in mind that $f_\eta$ is of linear growth, one gets
\begin{equation} \label{ke-big}
\zeta^2_x \mathcal{K}(\sigma) \leq \frac{C ||\zeta_x||_\infty^2}{k} (\sigma - k)_+^2.
\end{equation}

Using estimates \eqref{ke-big} and \eqref{ke-small} and careful multiplying by $-1$ lead to the conclusion that
\begin{equation*}
- \int_I (\zeta^2 G_k(u_x))_x \cdot D F_\eta (u_x)_x \dd \Lb^1 \leq \frac{C||\zeta_x||_\infty^2}{k} \left(  \int_I (\ourphi(u_x) - k)_+ \dd \Lb^1 + \eta  \int_I (\ourphi(u_x) - k)_+^2 \dd \Lb^1 \right).
\end{equation*}

Eventually, recall that if $\eps$ is fixed, then $\left \{ \int_I (\ourphi(u_x^{\eps, \eta}) - k)_+^2 \dd \Lb^1 \right \}$ is bounded w.\,r.\,t.\ $\eta$ (which was essentially proved in Lemma \ref{lem: cpt eta}), hence the second term vanishes when $\eta$ converges to zero. Lastly, the fact that 
\[
\E_{\eps, \eta}(u^{\eps, \eta}) \leq \int_I |h^\eps|^2 \dd \Lb^1 \leq \int_I |h|^2 \dd \Lb^1
\]
and linear growth of $f$ imply that $\left \{\int_I (\ourphi(u_x^{\eps, \eta})  - k)_+ \dd \Lb^1 \right\}$ is bounded (w.\,r.\,t.\ $\eta$ and $\eps$).
\end{proof}

\begin{proof}[Proof of Theorem \ref{thm:constants}]
	Recalling \eqref{eq: after mult} and \eqref{def_R}, Proposition \ref{prop: general lhs eps} imply that the minimizer $u$ of $\E$ satisfies
	\begin{equation} \label{eq:constants final}
		\frac{r(k)}{ (c_\varphi^+)^{3}} \int_I \zeta^2 \dd (c_\varphi^- |u_x| - k)_+(c_\varphi^+ |u_x| - k)_+ |u_x|^{-1} \leq  \frac{f^\infty}{ c_\varphi^- c_\varphi^+} \int_I \zeta^2 \dd (c_\varphi^+ |h_x| - k)_+ + \lambda \limsup_{\eps \to 0} \mathcal{R}(\eps, k),
	\end{equation}
where $\lim_{k \to \infty} r(k) = f^\infty$. Let us now pass with $k \to \infty$. By definition of measure $(c_\varphi^+|h_x| - k)_+$,
\[
	\int_I \zeta^2 \dd (c_\varphi^+|h_x| - k)_+ = \int_I \zeta^2 (c_\varphi^+|h_x^{ac}| - k)_+ \dd \Lb^1 + c_\varphi^+ \int_I \zeta^2 \dd |h_x^s|. 
\]
As $\zeta^2 |h_x^{ac}|$ is an integrable function, we can use the dominated convergence theorem to pass with the limit in the first integral on the r.\,h.\,s.\ and see that it vanishes. Therefore, by Proposition \ref{prop-rhs-eps-eta}, the r.\,h.\,s.\ of \eqref{eq:constants final} converges to $f^\infty(c_\varphi^-)^{-1} \int_I \zeta^2 \dd |h_x^s|$. Similarly, we have 
\begin{multline*}\int_I \zeta^2 \dd (c_\varphi^- |u_x| - k)_+(c_\varphi^+ |u_x| - k)_+ |u_x|^{-1} = \int_I \zeta^2 (c_\varphi^- |u_x^{ac}| - k)_+(c_\varphi^+ |u_x^{ac}| - k)_+ |u_x^{ac}|^{-1} \dd \Lb^1\\ + c_\varphi^- c_\varphi^+ \int_I \zeta^2 \dd |u_x^s|.
\end{multline*}
Thus, limit passage on the l.\,h.\,s.\ of \eqref{eq:constants final} can be performed in analogous manner, resulting in
\begin{equation} \label{eq:constants final2}
	\int_I \zeta^2 \dd |u_x^s| \leq \left ( \frac{c_\varphi^+}{c_\varphi^-} \right)^2 \int_I \zeta^2 \dd |h_x^s|.
\end{equation}
To finish the proof, it suffices to apply Lemma \ref{lem: mu leq nu} to $\mu = |u_x^s|$ and $\nu = (c_\varphi^+)^2(c_\varphi^-)^{-2}|h_x^s|$.
\end{proof}

\section{The regular case} \label{sec:regular} 

Firstly, let us note that under assumptions of Theorem~\ref{thm:regular norm}, $F$ is differentiable in $\R^n$. Additionally, any differentiable convex function on $\R^n$ is in fact continuously differentiable (\cite[Corollary 25.5.1]{rockafellar}). Since $\R^n \setminus \{0\}$ can be easily covered with a~finite number of open convex sets, this fact remains true if one replaces $\R^n$ with $\R^n \setminus \{0\}$, which means that $\varphi \in C^1(\R^n \setminus \{0\})$. In view of Lemma \ref{lem: strict convexity}, $F$ is strictly convex and the assumption of strict convexity in the sense of Reshetnyak on $\varphi$ is necessary for $F$ to be strictly convex.

We define $\F_M \colon M(I, \R^n) \to [0, \infty[$ by 
\[\F_M(\mu) = F(\mu)(I).\]
Clearly, we have $\F(w) = \F_M(w_x)$ for $w \in BV(I, \R^n)$.  

\begin{prop} \cite[Theorem 2.4]{anzellotti2} \label{thm:anzellotti}
	Functional $\F_M$ is differentiable at $\alpha \in M(I, \R^n)$ in direction $\beta \in M(I, \R^n)$ if and only if $|\beta^s|$ is absolutely continuous w.\,r.\,t.\ $|\alpha^s|$, and the derivative is of the form
\begin{equation*}
	\left.\frac{\dd}{\dd t}\right|_{t=0}\F_M(\alpha + t \beta) = \int_I DF(\alpha^{ac}) \cdot \beta^{ac} \dd \Lb^1 + \int_I DF^{\infty} \left( \frac{\alpha^s}{|\alpha^s|} \right) \cdot \dd \beta^s,
\end{equation*}
where $F^\infty(p) = \lim_{t \to \infty} \tfrac{1}{t}F(tp)$.
\end{prop}

By $W_0^{1,2}(I, \R^n)$ we denote the space of $W^{1,2}(I, \R^n)$ functions that vanish at $\partial I$. In this Section, we also chose $F_\eta$ to be defined as in \eqref{Feta2def}.
\begin{lemma} \label{lem: vanish at endpoints}
	The following statements hold:
	\begin{enumerate}[(i)]
		\item $DF_\eta(u_x^{\eps, \eta}) \in W^{1,2}_0(I, \R^n)$ and $\lambda \left( DF_\eta(u_x^{\eps, \eta}) \right)_x = u^{\eps, \eta} - h^\eps$ a.\,e.\ in $I$; 
		\item $DF(u_x^{\eps}) \in W^{1,2}_0(I, \R^n)$ and $\lambda \left( DF(u_x^{\eps}) \right)_x = u^\eps - h^\eps$ a.\,e.\ in $I$;
		\item $DF(u_x^{ac}) \in W^{1,2}_0(I, \R^n)$ and $\lambda \left( DF (u_x^{ac}) \right)_x = u - h$ a.\,e.\ in $I$.
	\end{enumerate}
\end{lemma}

\begin{proof}
We will show only proof of (iii) as it is the most interesting one. Statements (i) and (ii) are proved in the same way only without invoking the additional Proposition \ref{thm:anzellotti}.

Firstly, observe that since $DF$ is bounded, the function $DF(u_x^{ac}) \in L^2(I, \R^n)$. Secondly, applying Proposition \ref{thm:anzellotti} allows us to write the derivative of $\F_M$ at point $u_x^{ac}$ in direction $\xi \in C^\infty_c(I, \R^n)$ as 
\begin{equation*}
	\int_I \xi_x DF(u_x^{ac})  \dd \Lb^1,
\end{equation*}
since the integral with respect to the singular part vanishes. As $u$ is the minimizer of $\E$, it satisfies the Euler-Lagrange equation
\begin{equation*}
	- \lambda  \int_I \xi_x DF(u_x^{ac}) \dd \Lb^1 = \int_I \xi (u - h) \dd \Lb^1.
\end{equation*}
This means that the distributional derivative of $DF(u_x^{ac})$ coincides with $(u - h)/\lambda \in L^2(I, \R^n)$. Consequently, $DF(u_x^{ac}) \in W^{1,2}(I, \R^n)$, which implies also that there exists a~continuous representative of this function.
	
	We can show that any continuous representative equals zero at the endpoints of the interval by the following standard argument. Let $\xi^\ell(x) = (-\ell x + 1 + a \ell) \cdot \indi_{[a, a + 1/\ell]}$ be a~piecewise linear function equal $1$ for $x = a$ and $0$ on $[a, a + 1/\ell]$. Once again the Euler--Lagrange equation given by Proposition \ref{thm:anzellotti} yields 
\begin{equation*}
	-\lambda \int_I \xi^\ell_x DF(u_x^{ac})  \dd \Lb^1 = \int_I \xi^\ell (u - h)  \dd \Lb^1.
\end{equation*}
As $\xi^\ell_x = -\ell \indi_{[a, a + 1/\ell]}$, continuity of $DF(u_x^{ac})$ implies that the l.\,h.\,s. converges to $DF(u_x^{ac}(a))$ as $\ell \to \infty$. Moreover, the dominated convergence theorem enables us to pass with the limit under the integral sign in order to get zero on the r.\,h.\,s, which implies that $DF(u_x^{ac}(a)) = 0$. To get $DF(u_x^{ac}(b)) = 0$, one needs to repeat this argument with $\xi^\ell = (\ell x + 1 -\ell b) \cdot \indi_{[b-1/\ell, b]}$.
\end{proof}

Observe that for fixed representatives of $u^{\eps, \eta}$ and $DF_\eta(u_x^{\eps, \eta})$ it is true that almost everywhere on $I$, $DF_\eta(u_x^{\eps, \eta})$ is an actual composition of $DF$ and $u_x^{\eps, \eta}$, i.\,e., $DF_\eta(u_x^{\eps, \eta})(x_0) = DF_\eta(u_x^{\eps, \eta}(x_0))$. The same holds for analogous pairs of mappings from Lemma \ref{lem: vanish at endpoints}.

\subsection*{Limit passage with $\eta$}

Owing to the additional regularity of $F$, it is possible to perform a~more subtle limit passage with $\eta \to 0$ on the l.\,h.\,s.\ than in the previous case.

\begin{lemma} \label{lem: a.e. conv eta}
There exists a sequence $\eta_j$ with $\eta_j \to 0$ as $j \to \infty$ such that $u_x^{\eps, \eta_j}$ converges to $u_x^\eps$ a.\,e.\ on $I$.
\end{lemma}

\begin{proof}
To begin with, we will show that the sequence $DF_{\eta}( u_x^{\eps, \eta})$ converges to $DF(u_x^\eps)$ in $C( \overline{I}, \R^n)$. Due to the choice of $h^\eps$ and properties established in Lemma \ref{lem: cpt eta}, functions $ u^{\eps, \eta} - h^\eps$ converge in $L^2$ to $u^{\eps}-h^{\eps}$. In light of Lemma \ref{lem: vanish at endpoints}, it is then evident that $(DF_{\eta}(u_x^{\eps, \eta_j}))_x$ converge to $(DF(u_x^{\eps}))_x$ in $L^2$ as well. For every $x_0 \in \overline{I}$, due to the boundary condition established in the same lemma, 
\begin{equation*}
	DF_{\eta} (u_x^{\eps, \eta})(x_0) = \int_a^{x_0} \left(DF_{\eta} (u_x^{\eps, \eta}) \right)_x \dd \Lb^1 \quad \text{and} \quad DF (u_x^{\eps})(x_0) = \int_a^{x_0} \left(DF(u_x^{\eps}) \right)_x \dd \Lb^1.
\end{equation*}
The $L^2$ convergence of the integrands allows to conclude that
\begin{equation} \label{eq: a.e. conv eta}
	\sup_{x_0 \in \overline{I}} \left \lvert DF_{\eta}( u_x^{\eps, \eta}) (x_0) - DF(u_x^{\eps})(x_0) \right \rvert \to 0.
\end{equation}

Secondly, let us observe that
\begin{align} \label{eq: l1 convergence}
\begin{split}
	\int_I \left| DF_{\eta} (u_x^{\eps, \eta}) - DF(u_x^{\eps, \eta}) \right| \dd \Lb^1 &\leq \int_I \left| D(f_{\eta} \circ \ourphi)(u_x^{\eps, \eta}) - D(f \circ \varphi)(u_x^{\eps, \eta}) \right | \dd \Lb^1 +\\
	&+ \frac{\eta}{2} \int_I |D \widetilde{\widetilde{\varphi}}_{\eta} (u_x^{\eps, \eta})| \dd \Lb^1.
\end{split}
\end{align}
By Proposition \ref{goodnorm} (iv), we can see that the first integral converges to zero. The second one (up to a negligible constant coming from the fact that $\widetilde{\widetilde{\varphi}}_\eta$ coincides with $\widetilde{\varphi}_{\eta}$ outside a small ball) does not exceed
\[
	\eta \int_I \ourphi(u_x^{\eps, \eta}) |D\ourphi(u_x^{\eps, \eta}) | \dd \Lb^1.
\]
As $|D\ourphi|$ is a~bounded function and $\int_I \ourphi(u_x^{\eps, \eta}) \dd \Lb^1$ is bounded w.\,r.\,t.\ $\eta$ (see proof of Proposition \ref{prop-rhs-eps-eta}), we conclude that \eqref{eq: l1 convergence} converges to zero. This implies the~existence of a~subsequence $\eta_j$ with $\eta_j \to 0$ as $j \to \infty$ for which
\begin{equation*}
	\lim_{j \to \infty} \left| DF_{\eta_{j}} (u_x^{\eps, \eta_j}) - DF(u_x^{\eps, \eta_j}) \right| = 0 \text{ a.\,e.\ on } I.
\end{equation*}
Consequently, using \eqref{eq: a.e. conv eta} yields
\begin{equation*}
	\lim_{j \to \infty} \left| DF (u_x^{\eps, \eta_j}) - DF(u_x^{\eps}) \right| = 0 \text{ a.\,e.\ on } I.
\end{equation*}

As $F$ is strictly convex and $C^1$, $DF$ is strictly monotone and continuous. Strict monotonicity of $DF$ implies that the map is injective and hence, by Brouwer's invariance of domain theorem, there exists its continuous inverse $DF^{-1}$. Therefore, at any point in which $DF(u_x^{\eps, \eta_j})(x_0) \to DF(u_x^{\eps})(x_0)$ and the mappings $DF(u_x^{\eps, \eta_j})$ and $DF(u_x^{\eps})$ are actual compositions,
we have
\begin{equation*}
	u_x^{\eps, \eta_j}(x_0) = DF^{-1} \left( DF(u_x^{\eps, \eta_j}(x_0)) \right) \to DF^{-1} \left( DF(u_x^{\eps} (x_0)) \right) = u_x^{\eps}(x_0).
\end{equation*}
Since the set of such points is of full measure, we have established the desired a.\,e.\ convergence of $u_x^{\eps, \eta_j}$ to $u_x^\eps$.

\end{proof}

\begin{prop} \label{prop-lhs-eps-eta}
	There exists a sequence $\eta_j$ with $\eta_j \to 0$ as $j \to \infty$ for which the l.\,h.\,s.\ of \eqref{eq: after mult} satisfies
	\begin{equation*}
		\lim_{j \to \infty} \int_I \zeta^2 G_k(u_x^{\eps, \eta_j}) \cdot  (u_x^{\eps, \eta_j} - h_x^{\eps}) \dd \Lb^1
		= \int_I \zeta^2 \frac{(\varphi(u_x^\eps) - k)_+}{\varphi^2(u_x^{\eps})} f'(\varphi(u_x^\eps)) (u_x^{\eps} - h_x^{\eps}) \cdot u_x^\eps \dd \Lb^1.
	\end{equation*}
\end{prop}

\begin{proof}
Let us choose the sequence from Lemma \ref{lem: a.e. conv eta}, for which we know that the integrand is convergent a.\,e. Then, it suffices to check that it is uniformly integrable and apply the Vitali convergence theorem. Bearing in mind the formula for $G_k(\cdot)$, the absolute value of the integrand can be estimated from above by $C |u_{x}^{\eps, \eta} - h_x^{\eps}|$ for some constant $C$. Thus, uniform integrability follows from boundedness of the sequence $\{u_x^{\eps, \eta}\}_\eta$ in $L^2$. 
\end{proof}

\begin{prop} \label{prop: inequality with eps}
	The following estimate holds
	\begin{equation*}
		f'(k) \int_I \zeta^2 |u_x^\eps|^2 \frac{(\varphi(u_x^\eps) - k)_+}{\varphi^2(u_x^\eps)} \dd \Lb^1 \leq f^\infty \int_I \zeta^2 |h_x^\eps| |u_x^\eps| \frac{(\varphi(u_x^\eps) - k)_+}{\varphi^2(u_x^\eps)} \dd \Lb^1 + \lambda \mathcal{R}(\eps, k),
	\end{equation*}
where $\mathcal{R}(\eps, k)$ was defined in Proposition \ref{prop-rhs-eps-eta}.
\end{prop}

\begin{proof}
The proof boils down to observing what has already been proved about the Euler--Lagrange equation \eqref{Euler-Lagrange-eps-eta} of $\E_{\eps, \eta}$. Combining Propositions \ref{prop-rhs-eps-eta} and \ref{prop-lhs-eps-eta} gives
\begin{equation*}
	\int_I \zeta^2 \frac{(\varphi(u_x^\eps) - k)_+}{\varphi^2(u_x^\eps)} f'(\varphi(u_x^\eps)) |u_x^{\eps}|^2\dd \Lb^1 
	\leq
	 \int_I \zeta^2 \frac{(\varphi(u_x^\eps) - k)_+}{\varphi^2(u_x^\eps)} f'(\varphi(u_x^\eps)) h_x^{\eps} \cdot u_x^\eps \dd \Lb^1 + \mathcal{R}(\eps, k).
\end{equation*}
Then, on the l.\,h.\,s.\ we use the facts that $f'$ is non-decreasing and that the integration takes place on the set where $\varphi(u_x^\eps) > k$, whereas on the r.\,h.\,s.\ the fact that $f'(t) \leq f^\infty$ for all~$t$.
\end{proof}

\subsection*{Limit passage with $\eps$}

\begin{lemma} \label{lem: u eps a.e.}
There exist functions $\overline{u}_x^{\eps}, \overline{u}_x^{ac} \in C(\overline{I}, \overline{\R^n})$ s.\,t.\ $\overline{u}_x^{\eps} \to \overline{u}_x^{ac}$ in $C(\overline{I}, \overline{\R^n})$ and $\overline{u}_x^{\eps}(x_0) = u_x^{\eps}(x_0)$, $\overline{u}_x^{ac}(x_0) = u_x^{ac}(x_0)$ for $\Lb^1$-a.\,e.\ $x_0 \in \overline{I}$.
\end{lemma}
\begin{proof}
	Lemma \ref{lem: vanish at endpoints} implies (by reasoning similar to the beginning of the proof of Lemma \ref{lem: a.e. conv eta}) that $DF(u_x^\eps)$ converge to $DF(u_x^{ac})$ in $C(\overline{I}, \R^n)$ and therefore in $C(\overline{I}, \overline{\R^n})$. The map $\overline{DF}$ is a~homeomorphism onto its image, as established in Lemma \ref{lem: DF homeo}, which prompts the following definition:
\begin{equation*}
	\overline{u}_x^{\eps} := \overline{DF}^{-1} \circ DF(u_x^{\eps}) \text{ and } \overline{u}_x^{ac} := \overline{DF}^{-1} \circ DF (u_x^{ac}).
\end{equation*}
Requested properties of these functions now follow from continuity of $\overline{DF}^{-1}$ and the fact that $u_x^{\eps_i}$ and $u_x^{ac}$ are almost everywhere finite.
\end{proof}

\begin{proof}[Proof of Theorem \ref{thm:regular norm}]
Let us recall the inequality proved in Proposition \ref{prop: inequality with eps}
\begin{equation*}
		L_{\eps, k} :=  f'(k) \int_I \zeta^2 |u_x^\eps|^2 \frac{(\varphi(u_x^\eps) - k)_+}{\varphi^2(u_x^\eps)} \dd \Lb^1 \leq f^\infty \int_I \zeta^2 |h_x^\eps| |u_x^\eps| \frac{(\varphi(u_x^\eps) - k)_+}{\varphi^2(u_x^\eps)} \dd \Lb^1 + \lambda \mathcal{R}(\eps, k) := R_{\eps, k},
\end{equation*}
where $\zeta \in C_c^1(I)$ and $\lim_{k \to \infty} \limsup_{\eps \to 0} \mathcal{R}(\eps, k) = 0$ as stated in Proposition \ref{prop-rhs-eps-eta}. We will firstly pass with $\eps$ to zero. For $k >0$, we define a~continuous function $b_k\colon \overline{\R^n} \to \R$ by
\begin{equation*}
	b_k(p) = \begin{cases}
        0 & \text{ if } p = 0, \\
		\frac{|p| \left( \varphi(p) - k \right)_+}{\varphi^2(p)} & \text{ if } p \in \R^n \setminus \{0\},  \\
		\frac{1}{\varphi(\omega)} & \text{ if } p = \infty \omega \in \mathbb{S}^{n-1}_\infty.
	\end{cases}
\end{equation*}
Using Lemma \ref{lem: u eps a.e.}, we can rewrite
\begin{equation*}
	L_{\eps, k} = f'(k) \int_I \zeta^2 \left( b_k(\overline{u}_x^\eps) - b_k(\overline{u}_x^{ac}) \right) |u_x^{\eps}|\dd \Lb^1 + f'(k) \int_I \zeta^2 b_k(\overline{u}_x^{ac}) |u_x^{\eps}| \dd \Lb^1.
\end{equation*}
Again in view of Lemma \ref{lem: u eps a.e.}, we see that $b_k(\overline{u}_x^\eps)$ converges to $b_k(\overline{u}_x^{ac})$ in $C(\overline{I}, \R)$, which implies that the first term converges to zero (recall that $\{u_x^\eps\}_\eps$ is bounded in $L^1$). The second term is lower semicontinuous w.\,r.\,t.\ weak* convergence of $u_x^\eps$ to $u_x$ (see Corollary \ref{cor-gs-lsc-zeta}) and so
\begin{equation*}
	\lim_{\eps \to 0} L_{\eps, k} \geq f'(k) \int_I \zeta^2 b_k(\overline{u}_x^{ac}) \dd |u_x|.
\end{equation*}
As $h^\eps$ was chosen so that $|h_x^\eps| \weaklystar |h_x|$, one shows in a~similar manner that
\begin{equation*}
	\limsup_{\eps \to 0} R_{\eps, k} \leq f^\infty \int_I \zeta^2 b_k(\overline{u}_x^{ac}) \dd |h_x| + \limsup_{\eps \to 0} \mathcal{R}(\eps, k).
\end{equation*}
As a result,
\begin{align} \label{eq: before k limit}
	f'(k) \int_I \zeta^2 b_k(\overline{u}_x^{ac}) \dd |u_x^{s}| 		& \leq f^\infty \int_I \zeta^2 b_k(\overline{u}_x^{ac}) \dd |h_x| + \limsup_{\eps \to 0} \mathcal{R}(\eps, k) \leq \\
	& \leq f^\infty \int_I \zeta^2 b_k(\overline{u}_x^{ac})  |h_x^{ac}| \dd \Lb^1 + f^\infty \int_I \zeta^2 b_k(\overline{u}_x^{ac}) \dd |h_x^{s}| + \limsup_{\eps \to 0} \mathcal{R}(\eps, k). \nonumber
\end{align}

We will now pass with $k$ to infinity. At any point $x \in I$ for which $\overline{u}_x^{ac}(x) \in \mathbb{S}^{n-1}_\infty$, set $\omega(x) \in \mathbb{S}_{1}^{n-1}$ to be such that $\overline{u}_x^{ac}(x) = \infty \, \omega(x)$. Clearly, for any $x \in I$
\[
	b_k(\overline{u}_x^{ac}(x)) \xrightarrow{k \to \infty} b(x) := \frac{1}{\varphi(\omega(x))} \indi_{ \left\{x: \, \overline{u}_x^{ac}(x) \in \mathbb{S}_\infty^{n-1} \right\}}(x).
\]
Therefore, the fact that $|\overline{u}_x^{ac}|$ is a.\,e.\ finite together with dominated convergence theorem allow us to conclude that \eqref{eq: before k limit} after passing with $k$ to infinity and dividing by $f^\infty$ becomes
\begin{equation}\label{eq: after k limit}
	\int_I \zeta^2 b \dd |u_x^s| \leq \int_I \zeta^2 b \dd |h_x^s|.
\end{equation}
Denoting  $\dd \mu =  b \dd |u_x^s|$, $\dd \nu =  b \dd |h_x^s|$ we deduce $\mu \leq \nu$ as Borel measures by Lemma \ref{lem: mu leq nu}, and therefore also $\widetilde{\mu} \leq \widetilde{\nu}$, where 
\[\dd \widetilde{\mu} \!:= \varphi(\omega) \dd \mu =  \indi_{ \left \{x: \, \overline{u}_x^{ac}(x) \in \mathbb{S}_\infty^{n-1} \right\}} \dd |u_x^s|, \quad \dd \widetilde{\nu} \!:= \varphi(\omega) \dd \nu =  \indi_{ \left \{x: \, \overline{u}_x^{ac}(x) \in \mathbb{S}_\infty^{n-1} \right\}} \dd |h_x^s|\]
By Lemma \ref{lem: ac sphere} below, $\widetilde{\mu}$ actually coincides with $|u_x^s|$. On the other hand, clearly $\widetilde{\nu} \leq |h_x^s|$, which concludes the proof. 
\end{proof}

\begin{lemma} \label{lem: ac sphere}
	For $|u_x^s|$-a.\,e.\ $x$ it is true that $\overline{u}_x^{ac}(x) \in \mathbb{S}_\infty^{n-1}$.
\end{lemma}

\begin{proof}
Choose $x_0$ such that $|\overline{u}_x^{ac}(x_0)| < \infty$. Since $\overline{u}_x^{ac}$ is continuous, there exists an open interval $J \subset I$ such that for any $y \in J$ it is true that $\overline{u}_x^{ac}(y) \in \R^n$. Then for any $\xi \in C_c(J)$,
\begin{equation*}
	\left | \int_J u_x^{\eps} \xi \dd \Lb^1 - \int_J u_x^{ac} \xi \dd \Lb^1 \right | = \left | \int_J \overline{u}_x^{\eps} \xi \dd \Lb^1 - \int_J \overline{u}_x^{ac} \xi \dd \Lb^1 \right | \leq \sup_{J} |\overline{u}_x^\eps - \overline{u}_x^{ac} | \int_J \xi \dd \Lb^1 \xrightarrow{\eps \to 0} 0,
\end{equation*}
due to the facts that $\overline{u}_x^{\eps}$ coincides with $u_x^\eps$ $\Lb^1$-a.\,e., $\overline{u}_x^{ac}$ coincides with $u_x^{ac}$ $\Lb^1$-a.\,e.\ and that $\overline{u}_x^{\eps}$ converge in $C(\overline{I}, \overline{\R^n})$ to $\overline{u}_x^{ac}$. Thus we have shown that $u_x^\eps \weaklystar u_x^{ac}$ on $J$. On the other hand, Lemma \ref{lem: gamma conv eps} says that $u_x^\eps \weaklystar u_x$ on $J$, which implies that $u_x^s \mres J = 0$ and, consequently, that
\[
	|u_x^s| \mres J = 0,
\]
which in turn means that $x_0 \notin \mathrm{supp}\,|u_x^s|$ (the support of the measure $|u_x^s|$). We have therefore established that if $x_0 \in \mathrm{supp}\,|u_x^s|$, then $\overline{u}_x^{ac} \in \mathbb{S}_\infty^{n-1}$. Recalling that the set $\mathrm{supp}\,|u_x^s|$ is of full $|u_x^s|$ measure we conclude the proof. \\
\end{proof} 

\noindent 
{\bf \Large Acknowledgements} \\ 

The second author was partially supported by the grant no.\ 2020/36/C/ST1/00492 of the National Science Center, Poland. Part of this work was created during second author's JSPS Postdoctoral Fellowship at the University of Tokyo.  
\bibliographystyle{asdfgh-2}
\bibliography{bib.bib.bib}
\end{document}